\documentclass[12pt,a4paper,reqno]{amsart}

\usepackage[latin1]{inputenc}
\usepackage[english]{babel}

\usepackage{amsmath}
\usepackage{amsfonts}
\usepackage{amssymb}
\usepackage{amsthm}
\usepackage{mathtools}

\usepackage{graphicx}
\usepackage{enumitem}
\usepackage[final]{showkeys}

\usepackage{cases}

\usepackage[left=2cm,right=2cm,%
top=2cm,bottom=2cm]{geometry}

\usepackage{hyperref}
\hypersetup{%
colorlinks=true, %
linkcolor=blue, %
citecolor=blue, %
filecolor=blue, %
urlcolor=blue, %
pdfauthor={M. J. Alves, R. B. Assuncao and O. H. Miyagaki}, %
pdftitle={Existence results for a class of quasilinear elliptic equations with $(p,q)$-Laplacian and vanishing potentials}, %
pdfsubject={Elliptic partial differential equations}, %
pdfkeywords={Quasilinear elliptic equations, 
(p,q)-Laplacian operator, 
variational methods,
penalization method, 
vanishing potentials}, %
pdfproducer={Latex}, %
pdfcreator={ps2pdf}}

\newtheorem{thm}{Theorem}[section]
\newtheorem{lem}[thm]{Lemma}

\theoremstyle{remark}

\newtheorem{claim}{Claim}

\allowdisplaybreaks

\title[Existence result for ($p$-$q$)-laplacian with vanishing potentials]
{Existence result for a class of quasilinear elliptic equations with ($\mathbf{p}$-$\mathbf{q}$)-laplacian and vanishing potentials}

\author[Alves]{M.\@ J.\@ Alves}
\address{M.\@ J.\@ Alves \hfill\break\indent
Col\'{e}gio T\'{e}cnico\,---\,
Universidade Federal de Minas Gerais, UFMG \hfill\break\indent
Av.~Ant\^{o}nio Carlos, 6627\,---\,
CEP 31270-901\,---\,
Belo Horizonte, MG, Brasil}
\email{mariajose@ufmg.br}

\author[Assun\c{c}\~{a}o]{R.\@ B.\@ Assun\c{c}\~{a}o}
\address{R.\@ B.\@ Assun\c{c}\~{a}o \hfill\break\indent
Departamento de Matem\'{a}tica\,---\,
Universidade Federal de Minas Gerais, 
UFMG \hfill\break\indent
Av.~Ant\^{o}nio Carlos, 6627\,---
CEP 30161-970\,---\,Belo Horizonte, MG, Brasil}
\email{ronaldo@mat.ufmg.br}

\author[Miyagaki]{O.\@ H.\@ Miyagaki}
\thanks{O.\@ H.\@ Miyagaki was partially supported by 
CNPq/Brasil and INCTMAT/Brasil.}
\address{O.\@ H.\@ Miyagaki \hfill\break\indent
Departamento de Matem\'{a}tica\,---\,
Universidade Federal de Juiz de Fora, 
UFJF \hfill\break\indent
Cidade Universit\'{a}ria\,---\,
CEP 36036-330\,---\,Juiz de Fora, MG, Brasil} 
\email{ohmiyagaki@gmail.com}

\date{Belo Horizonte, \today}



\begin{document}

\begin{abstract}
The main purpose of this paper is to establish the 
existence of positive solutions
to a class of quasilinear elliptic equations involving 
the ($p$-$q$)-Laplacian operator. 
We consider a nonlinearity that can be subcritical at 
infinity and supercritical at the origin; 
we also consider potential functions that 
can vanish at infinity. 
The approach is based on variational arguments dealing 
with the mountain-pass lemma and an adaptation of the 
penalization method.
In order to overcome the lack of compactness we 
modify the original problem and the associated 
energy functional. 
Finally, to show that the solution of the modified 
problem is also a solution of the original 
problem we use an estimate obtained by the Moser 
iteration scheme.
\end{abstract}

\maketitle

\section{Introduction and main result}
\label{sec:intmain}
In this paper we consider a class of quasilinear 
elliptic equations involving the ($p$-$q$)-Laplacian 
operator of the form
\begin{align}
\label{eq:prob}
\begin{cases}
-\Delta\sb{p}u -\Delta\sb{q}u 
+ a(x) \left|u\right|\sp{p-2} u 
+ b(x) \left|u\right|\sp{q-2} u
= f(u), & \qquad x \in \mathbb{R}\sp{N};\\
\phantom{-}u(x) > 0, 
\quad u \in D\sp{1,p}(\mathbb{R}\sp{N})
\cap D\sp{1,q}(\mathbb{R}\sp{N}), 
& \qquad x \in \mathbb{R}\sp{N}.
\end{cases}
\end{align}
The \(m\)-laplacian operator \( \Delta\sb{m} u(x) \)
is defined by
\begin{align*}
\Delta\sb{m}u(x) \equiv 
\operatorname{div}(|\nabla u(x)|\sp{m-2}\nabla u(x)), 
\end{align*} 
for \( m \in \{ p, q \}\), where
\( 2 \leqslant q \leqslant p < N\).
The Sobolev space 
\( D\sp{1,m}(\mathbb{R}\sp{N}) \)
is defined by
\begin{align*}
D\sp{1,m}(\mathbb{R}\sp{N})
\equiv \big\{ u \in L\sp{m\sp{*}}(\mathbb{R}\sp{N}) 
\colon   
(\partial u/\partial x\sb{i})(x) 
\in L\sp{m}(\mathbb{R}\sp{N}),
\quad 1 \leqslant i \leqslant N \big\}, 
\end{align*}
and the critical Sobolev exponent is given by 
\(m\sp{*}\equiv Nm/(N-m)\), also for \( m \in \{ p,q \} \).

The nonlinearity \(f \colon \mathbb{R} \to \mathbb{R}\) 
is a continuous and nonnegative function that is not a 
pure power and can be subcritical at infinity and 
supercritical at the origin. More precisely, the 
following set of hypotheses on the nonlinearity 
\(f\) is used.

\begin{enumerate}[label=($f\sb{\arabic*}$), 
ref=$f\sb{\arabic*}$]
\item \( \limsup\sb{s \to 0\sp{+}} 
s f(s)/s\sp{p\sp{*}} < + \infty \).
\label{hyp:f1}
\item There exists \( \tau \in (p,p\sp{*}) \) such that 
\( \limsup\sb{s \to +\infty} s f(s)/s\sp{\tau} = 0 \).
\label{hyp:f2}
\item There exists \( \theta > p \) such that
\( 0 \leqslant \theta F(s) \leqslant s f(s) \)
for every \( s \in \mathbb{R}\sp{+} \), 
where we use the notation
\( F(s) \equiv \displaystyle 
\int\sb{0}\sp{s} f(t)\,\mathrm{d}t \). 
\label{hyp:f3}
\item \(f(t) = 0 \) for every \( t \leqslant 0 \). 
\label{hyp:f4}
\end{enumerate}

The following properties are easily seen: under 
hypothesis~\eqref{hyp:f1} there exists 
\( c\sb{1} \in \mathbb{R}\sp{+} \) such that
\( |sf(s)| \leqslant
c\sb{1} |s|\sp{p\sp{*}} \) for \( s \) close to zero; 
and under hypothesis~\eqref{hyp:f2}  
there exists \( c\sb{2}  \in \mathbb{R}\sp{+} \) 
such that
\( |sf(s)| \leqslant c\sb{2} |s|\sp{\tau} \) for
\( s \) large enough. 
Combining these results and defining
\( c\sb{0} \equiv \max \{ c\sb{1}, c\sb{2} \} \), 
we have the pair of inequalities
\begin{alignat}{2}
\label{eq:conshypf1f2}
|sf(s)| & \leqslant c\sb{0} |s|\sp{p\sp{*}}
\quad \text{and}
& \quad
|sf(s)| & \leqslant c\sb{0} |s|\sp{\tau}
\qquad (s \in \mathbb{R}).
\end{alignat}

Is is worth noticing that hypothesis~\eqref{hyp:f3} 
extends a well known condition which was first 
formulated by Ambrosetti and Rabinowitz~\cite{bib:ar}. 
It states a sufficient condition to ensure that the 
energy funcional, associated in a natural way to this 
type of problem, verifies the Palais-Smale condition.
Recall that a functional 
\( J \colon D\sp{1,m}(\mathbb{R}\sp{N}) 
\to \mathbb{R} \) 
is said to verify the Palais-Smale condition at the 
level \( c \) if any sequence 
\( (u\sb{n})\sb{n \in \mathbb{N}} 
\subset D\sp{1,m}(\mathbb{R}\sp{N})\) 
such that 
\( J(u\sb{n}) \to c \) and \( J'(u\sb{n}) \to 0 \), as 
\( n \to + \infty \), possess a convergent subsequence. 
Hypothesis~\eqref{hyp:f3} also allows us to study the 
asymptotic behavior of the solution to the problem.

As an example of a nonlinearity \( f \) verifying the 
above set of hypotheses, for \( \sigma > p\sp{*} \) 
and for \( \tau \in (p,p\sp{*})\)  given in 
hypothesis~\eqref{hyp:f2}, we define
\begin{align*}
f(t) & = 
\begin{cases}
t\sp{\sigma -1}, & 
\text{if $0 \leqslant t \leqslant 1$;} \\
t\sp{\tau -1}, & \text{if $1 \leqslant t$.}
\end{cases}
\end{align*}

We also assume that the functions 
\(a,b \colon \mathbb{R}\sp{N} \to \mathbb{R}\) 
are continuous and nonnegative.
Moreover, the following set of hypotheses on the 
potential functions 
\(a\) and \( b \) is used. 
\begin{enumerate}[label=($P\sb{\arabic*}$), 
ref=$P\sb{\arabic*}$]
\item  
\label{hyp:v1}
\(a \in L\sp{N/p}(\mathbb{R}\sp{N})\)
and  
\(b \in L\sp{N/q}(\mathbb{R}\sp{N})\).

\item  
\label{hyp:v2}
\( a(x) \leqslant a\sb{\infty} \) and 
\( b(x) \leqslant b\sb{\infty} \) for every 
\( x \in B\sb{1}(0) \), 
where 
\( a\sb{\infty}, b\sb{\infty} \in \mathbb{R}\sp{+} \) 
are positive constants and 
\( B\sb{1}(0) \) denotes the unitary ball centered at 
the origin. 

\item 
\label{hyp:v3}
There exist constants \( \Lambda \in \mathbb{R}\sp{+} \) 
and
\( R_0 > 1 \) such that
\begin{equation*}
\dfrac{1}{R_0\sp{p \sp{2}/(p-1)}} 
\inf\sb{|x| \geqslant R_0} 
|x|\sp{p\sp{2}/(p-1)} a(x) \geqslant \Lambda.
\end{equation*}  
\end{enumerate}

As an example of a potential function \( a \) verifying
this set of hypotheses, for 
\( \Lambda \in \mathbb{R}\sp{+} \) and \( R\sb{0} > 1 \)
given in hypothesis~\eqref{hyp:v3} we define
\begin{align*}
a(x) & = 
\begin{cases}
0, & \text{if $|x| \leqslant R\sb{0} - 1$;} \\
\Lambda R\sb{0}\sp{-p\sp{2}/(p-1)}(|x| - R\sb{0}+1),
& \text{if $R\sb{0} - 1 < |x| < R\sb{0}$;} \\
\Lambda |x|\sp{-p\sp{2}/(p-1)},
& \text{if $R\sb{0} \leqslant |x|$.} 
\end{cases}
\end{align*}
An example of a potential function \( b \) can be 
obtained in a similar way with minor modifications.

The ($p$-$q$)-Laplacian operator generalizes several 
types of problems. For example,
in the case \( 2 = q = p \) 
with \( a(x) = b(x) = V(x) \) and 
\( f(u) = 2g(u)\),
problem~\eqref{eq:prob} can be written in the form
\( -\Delta u + V(x)u = g(u) \),
which appears in the study of stationary solutions of 
Schr\"{o}dinger equation and has been extensively 
studied by several authors; 
and in the case \( 2 \leqslant q = p \)
with \( a(x) = b(x) = -V(x) \)
and \( f(u) = 0 \), 
problem~\eqref{eq:prob} assumes the form of the 
eigenvalue problem 
\( -\Delta\sb{p}u = V(x)\vert u \vert\sp{p-2}u \).

The interest in the study of this type of problem is 
twofold. On the one hand we have the physical 
motivations, since the quasilinear operator ($p$-$q$)-
Laplacian has been used to model steady-state solutions
of reaction-diffusion problems arising in biophysics, 
in plasma physics and 
in the study of chemical reactions. 
More precisely, the prototype for these models can
be written in the form
\begin{align*}
u\sb{t} & = -\operatorname{div}[D(u)\nabla u] + f(x,u),
\end{align*}
where 
\( D(u) = a\sb{p} |\nabla u|\sp{p-2} 
        + b\sb{q} |\nabla u|\sp{q-2} \) 
and \( a\sb{p}, b\sb{q} \in \mathbb{R}\sp{+}\) 
are positive constants.
In this framework, 
the function \( u \) generally stands 
for a concentration, 
the term \( \operatorname{div}[D(u)\nabla u] \) 
corresponds to the diffusion 
with coefficient \( D(u) \), 
and \( f(x,u) \) is the reaction term related 
to source and loss processes. 
See Cherfils and Il'yasov~\cite{bib:ci}, 
Figueiredo~\cite{bib:fig11,bib:fig13}, 
Benouhiba and Belyacine~\cite{bib:bb}, 
Mercuri and Squassina~\cite{bib:ms}, 
Wu and Yang~\cite{bib:wy}, 
Yin and Yang~\cite{bib:yy},
Chaves, Ercole and Miyagaki~\cite{bib:cembvp,bib:cemna},
and references therein for more details.
In addition, a model of 
elementary particle physics was studied by 
Benci, D'Avenia, Fortunato and Pisani~\cite{bib:bafp}
which yields an equation of the same class as that in 
problem~\eqref{eq:prob}.

On the other hand we have the 
purely mathematical interest in these type of problems, 
mainly regarding the existence of nonnegative nontrivial 
solutions  as well as multiplicity results. 
In what follows we present a very brief historical 
sketch to show some hypotheses on the nonlinearity that 
have been used by several authors in recent years as 
sufficient conditions to guarantee the existence of 
solutions.

We begin by considering the case
\( 2 \leqslant q = p < p\sp{*}\),  
which includes both the Laplacian operator with 
\( p = 2 \) 
or the \( p \)-Laplacian operator with \( p > 2 \); 
we also mention some 
papers dealing with bounded domains and others dealing 
with the entire space \( \mathbb{R}\sp{N} \).

Berestycki and Lions~\cite{bib:bl} considered 
a positive, constant potential function to show an 
existence result. 
Coti Zelati and Rabinowitz~\cite{bib:czr},
Pankov~\cite{bib:p},
Pankov and Pfl\"{u}ger~\cite{bib:pp},
and
Kryszewski and Szulkin~\cite{bib:ks}
considered periodic potential functions 
with a positive infimum.
Zhu and Yang~\cite{bib:yzamsi,bib:zyjpde} assumed that 
the potential is asymptotic to a positive constant.
Alves, Carri\~{a}o and Miyagaki~\cite{bib:coacm} 
studied a problem involving 
an asymptotically periodic potential. 
The case of a coercive potential was treated, 
among others, by 
Costa~\cite{bib:c} and Miyagaki~\cite{bib:m}.
For a weakened coercivity condition 
we refer the reader to 
Bartsch and Wang~\cite{bib:bw}.
The case of radially symmetric potentials were 
considered by 
Alves, de Morais Filho and Souto~\cite{bib:amfs} and 
Su, Wang and Willem~\cite{bib:sww}, 
where these authors established some embedding results 
of weighted Sobolev spaces to obtain ground state 
solutions.
Rabinowitz~\cite{bib:r} introduced a hypothesis where 
the limit inferior of the potential outside a bounded 
domain is strictly greater than its infimum on the whole 
space. Afterwards,
del Pino and Felmer~\cite{bib:dpf} 
weakened this condition by considering a situation where 
the minimum of the potential on the boundary of an open 
bounded set is strictly greater than its minimum on the 
closure of this set.
The case of sign-changing potentials related to singular 
perturbation problems were considered by 
Ding and Szulkin~\cite{bib:ds} and by 
Alves, Assun\c{c}\~{a}o, Carri\~{a}o and 
Miyagaki~\cite{bib:aacm}.

As we have seen, most of the papers cited assume that the potential is positive at infinity. 
However, the case where the potential can 
vanish at infinity was also studied, 
among others, by
Berestycki and Lions~\cite{bib:bl},
Yang and Zhu~\cite{bib:yzamsii},
Benci, Grisanti and Micheletti~\cite{bib:bgm}, 
Ambrosetti and Wang~\cite{bib:aw}, 
Ambrosetti, Felli and Malchiodi~\cite{bib:afm},
Alves and Souto~\cite{bib:as}, and 
Bastos, Miyagaki and Vieira~\cite{bib:bmv}.

In problem~\eqref{eq:prob} we consider the exponents
\( 2 \leqslant q \leqslant p < N \) 
and we allow the particular conditions
\( \liminf\sb{\vert x \vert \to +\infty} a(x) = 0\) and
\( \liminf\sb{\vert x \vert \to +\infty} b(x) = 0\),
called the zero mass cases.
These constitute the main features of our work. 

Our result reads as follows.

\begin{thm}
\label{thm:main}
Consider \( 2 \leqslant q \leqslant p < N \)
and suppose that the potential functions 
\( a \) and \( b \) verify the 
hypotheses~\eqref{hyp:v1},~\eqref{hyp:v2}
and~\eqref{hyp:v3}
and that the nonlinearity \( f \) verifies the 
hypotheses~\eqref{hyp:f1},~\eqref{hyp:f2},~\eqref{hyp:f3}, 
and~\eqref{hyp:f4}.
Then there exists a constant
\( \Lambda\sp{*} = \Lambda\sp{*}(a\sb{\infty}, 
b\sb{\infty},\theta, \tau, c\sb{0}) \) 
such that problem~\eqref{eq:prob} has a positive 
solution for every 
\( \Lambda \geqslant \Lambda\sp{*} \).
\end{thm}

Usually, a solution to problem~\eqref{eq:prob} 
is obtained as a critical point of the corresponding 
energy functional defined in some appropriate Sobolev 
space. To do this one uses critical point theory, 
mainly of minimax type; 
see Mawhin and Willem~\cite{bib:mw}, 
Struwe~\cite{bib:str}, and 
Willem~\cite{bib:w}. 
A well known result concerning the existence of a 
nontrivial weak solution is that if the 
energy functional verifies the 
geometry of the mountain-pass lemma near the origin 
and also verifies the Palais-Smale condition, 
then problem~\eqref{eq:prob} has at least one solution. 
The main difficulty in proving the existence of solution 
to problem~\eqref{eq:prob} resides in the fact that the 
embedding of the Sobolev space 
\( D\sp{1,m}(\mathbb{R}\sp{N}) \) 
in the Lebesgue space 
\( L\sp{Nm/(N-m)}(\mathbb{R}\sp{N})\) 
is not compact due to the action of a
group of homoteties and translations. 
Besides, the Palais-Smale condition for the 
corresponding energy functional 
cannot be obtained directly. 
Adding to these difficulties, 
we have to consider the presence of both operators 
\(\Delta\sb{p}u\) and \(\Delta\sb{q} u\). 
When \( q < p \) the study of problem~\eqref{eq:prob}
does not allow the use of the Lagrange's multipliers 
method due to the lack of homogeneity; 
moreover, the first eigenvalue of the 
\(-\Delta\sb{p}u\) operator brings
no valuable information on the  
eigenvalue of the \(-\Delta\sb{q}u\) operator; 
finally, the method of sub- and super-solutions 
cannot be applied.
Therefore, to study problem~\eqref{eq:prob} we are
required to make a careful analysis of the energy level 
of the Palais-Smale sequences in order to obtain their 
boundedness and also 
to overcome the lack of compactness.
Furthermore, we have to adapt the Moser iteration scheme 
to our setting, since this is a crucial step to obtain 
an estimate for the solution.

Inspired mainly by 
Wu and Yang~\cite{bib:wy} 
regarding the ($p$-$q$)-Laplacian type operator, 
and by 
Alves and Souto~\cite{bib:as}, 
with respect to the set of hypotheses, 
we adapt the penalization method developed by 
del Pino and Felmer~\cite{bib:dpf} 
to show our existence result. 
The basic idea can be described in the following way. 
In section~\ref{sec:auxprob} we modify the original 
problem and study its corresponding energy functional, 
showing that it verifies the 
geometry of the mountain-pass lemma 
and that every Palais-Smale sequence is bounded 
in an appropriate Sobolev space. 
Using the standard theory 
this implies that the modified problem has a solution. 
In section~\ref{sec:estimsol} we show, 
using the Moser iteration scheme, 
that the solution of the auxiliary problem verifies 
an estimate involving the 
\( L\sp{\infty}(\mathbb{R}\sp{N}) \) norm.
Finally, in section~\ref{sec:obtainsol} 
we use this estimate 
to show that the solution of the modified problem 
is also a solution of the original 
problem~\eqref{eq:prob}.

\section{An auxiliary problem}
\label{sec:auxprob} 
In order to prove the existence of a positive solution 
to problem~\eqref{eq:prob} 
we establish a variational setting and apply the 
mountain-pass lemma.
Using hypothesis~\eqref{hyp:v1}
we define the space
\begin{equation*}
E \equiv \Big\{ u \in 
D\sb{a}\sp{1,p}(\mathbb{R}\sp{N}) \cap 
D\sb{b}\sp{1,q}(\mathbb{R}\sp{N})\colon
\int\sb{\mathbb{R}\sp{N}} a(x) |u|\sp{p}\,\mathrm{d}x 
< + \infty  \text{\, and\,}
\int\sb{\mathbb{R}\sp{N}} b(x) |u|\sp{q}\,\mathrm{d}x 
< + \infty \Big\},
\end{equation*}
which can be endowed with the norm
\(\left\|u\right\| =\left\|u\right\|_{1,p}  
+ \left\|u\right\|_{1,q}  \),  
where we denote
\begin{align*}
\| u \| _{1,p} 
& \equiv \Big(
\int\sb{\mathbb{R}\sp{N}}
|\nabla u|\sp{p}\,\mathrm{d}x + 
\int\sb{\mathbb{R}\sp{N}}
a(x) |u|\sp{p}\,\mathrm{d}x 
\Big)\sp{1/p} \\
\shortintertext{and}
\| u \| _{1,q}
& \equiv \Big(
\int\sb{\mathbb{R}\sp{N}}
|\nabla u|\sp{q}\,\mathrm{d}x 
+ \int\sb{\mathbb{R}\sp{N}}
b(x) |u|\sp{q}\,\mathrm{d}x 
\Big)\sp{1/q}.
\end{align*}

The Euler-Lagrange energy functional 
\( I \colon E \to \mathbb{R} \)
associated to problem~\eqref{eq:prob} is defined by
\begin{align*}
I(u) &\equiv 
 \dfrac{1}{p} \int\sb{\mathbb{R}\sp{N}} 
| \nabla u|\sp{p}\,\mathrm{d}x
+ \dfrac{1}{p}\int\sb{\mathbb{R}\sp{N}} 
a(x) |u|\sp{p}\,\mathrm{d}x \\
& \qquad +\dfrac{1}{q} \int\sb{\mathbb{R}\sp{N}} 
|\nabla u|\sp{q}\,\mathrm{d}x
+ \dfrac{1}{q}\int\sb{\mathbb{R}\sp{N}} 
b(x) |u|\sp{q}\,\mathrm{d}x
- \int\sb{\mathbb{R}\sp{N}} F(u) \,\mathrm{d}x.
\end{align*}
Using the hypotheses on the nonlinearity \( f \) 
we can deduce that 
\( I \in C\sp{1}(E; \mathbb{R}) \); 
moreover, for every
\( u, v \in E \)
its G\^{a}teaux derivative can be computed by
\begin{align*}
I'(u) v &= 
\int\sb{\mathbb{R}\sp{N}} 
|\nabla u|\sp{p-2} \nabla u \cdot \nabla v \,\mathrm{d}x
+ \int\sb{\mathbb{R}\sp{N}} 
a(x) |u|\sp{p-2} u v \,\mathrm{d}x\nonumber\\
& \qquad + \int\sb{\mathbb{R}\sp{N}} 
|\nabla u|\sp{q-2} \nabla u \cdot \nabla v \,\mathrm{d}x
+ \int\sb{\mathbb{R}\sp{N}} 
b(x) |u|\sp{q-2}u v \,\mathrm{d}x
- \int\sb{\mathbb{R}\sp{N}} f(u) v \,\mathrm{d}x.
\end{align*}

It is a well known fact that if \( u \) 
is a critical point of the energy functional \( I \), 
then \( u \) is a weak solution to 
problem~\eqref{eq:prob}. This means that
\begin{align*} 
{} & \int\sb{\mathbb{R}\sp{N}} 
|\nabla u|\sp{p-2} 
\nabla u \cdot \nabla \phi \,\mathrm{d}x
+ \int\sb{\mathbb{R}\sp{N}} 
a(x) |u|\sp{p-2}u\phi\,\mathrm{d}x \\
& \qquad + \int\sb{\mathbb{R}\sp{N}} 
|\nabla u|\sp{q-2} \nabla u \cdot 
\nabla \phi \,\mathrm{d}x
+ \int\sb{\mathbb{R}\sp{N}} 
b(x) |u|\sp{q-2}u\phi\,\mathrm{d}x
- \int\sb{\mathbb{R}\sp{N}} f(u) \phi\,\mathrm{d}x = 0
\end{align*}
for every \( v \in E \).

Now we define the energy functional
\( I\sb{\infty} \colon 
D\sb{0}\sp{1,p}(B\sb{1}(0))
\cap 
D\sb{0}\sp{1,q}(B\sb{1}(0)) \to \mathbb{R} \) 
by
\begin{align*}
 I\sb{\infty}(u) & \equiv 
\dfrac{1}{p} \int\sb{B\sb{1}(0)} 
| \nabla u|\sp{p}\,\mathrm{d}x
+ \dfrac{1}{p}\int\sb{B\sb{1}(0)} 
a_{\infty} |u|\sp{p}\,\mathrm{d}x\\
&\qquad + 
\dfrac{1}{q} \int\sb{B\sb{1}(0)} 
| \nabla u|\sp{q}\,\mathrm{d}x
+ \dfrac{1}{q}\int\sb{B\sb{1}(0)} 
b_{\infty} |u|\sp{q}\,\mathrm{d}x
- \int\sb{B\sb{1}(0)} F(u) \,\mathrm{d}x.
\end{align*}
Using the hypotheses~\eqref{hyp:v1} and~\eqref{hyp:v2} 
it can be shown that it is well defined.
Our first lemma concerns the geometry of this 
functional. 

\begin{lem}
\label{lem:I}
The functional \(I_{\infty}\) verifies the 
geometry of the mountain-pass lemma. 
More precisely, the following claims are valid.
\begin{enumerate}
\item \label{mpti1} 
There exist 
\(r_0, \; \mu _0 \in \mathbb{R}\sp{+}\) such that  
\(I_{\infty}(u)\geqslant \mu_0\) 
for \(\left\|u\right\|=r_0\).
\item \label{mpti2}
There exists 
\(e_0 \in  
\left[ D\sp{1,p}(B\sb{1}(0))
\cap 
D\sp{1,q}(B\sb{1}(0)) \right] 
\backslash \{ 0 \} \) 
such that  
\(\left\| e_0 \right\| \geqslant r_0\) 
and \(I_{\infty}(e_0)<0\).
\end{enumerate}
\end{lem}

\begin{proof}
By using the hypotheses~\eqref{hyp:f1},~\eqref{hyp:f2}, 
and~\eqref{hyp:f3} it is standard to verify 
item~\eqref{mpti1}. 

By hypothesis~\eqref{hyp:f3} it follows that there exist 
$\theta > p$ and $C_{0} \in \mathbb{R}\sp{+}$
such that 
$F(s)\geqslant C_{0} \left|s\right|\sp{\theta}$.
Now, if 
\( u \in 
\left[ D\sb{0}\sp{1,p}(B\sb{1}(0))
\cap 
D\sb{0}\sp{1,q}(B\sb{1}(0)) \right] 
\backslash \{ 0 \} \),
then 
\begin{align*}
I\sb{\infty} (tu)& 
\leqslant 
\dfrac{1}{p} \left|t\right|\sp{p}\int\sb{B\sb{1}(0)} 
|\nabla u|\sp{p}\,\mathrm{d}x
+ \dfrac{a_{\infty}}{p}\left|t\right| \sp{p}
\int\sb{B\sb{1}(0)} |u|\sp{p}\,\mathrm{d}x\\
& \qquad + \dfrac{1}{q}\left|t\right|\sp{q}
\int\sb{B\sb{1}(0)} | \nabla u|\sp{q}\,\mathrm{d}x
+ \dfrac{b_{\infty}}{q}\left|t\right|\sp{q}
\int\sb{B\sb{1}(0)} |u|\sp{q}\,\mathrm{d}x
- C_{0}\left|t\right|\sp{\theta} 
\int\sb{B\sb{1}(0)}   
\left|u\right|\sp{\theta} \,\mathrm{d}x.
\end{align*} 
Using this inequality we deduce that 
there exist $t_{u} \in \mathbb{R}\sp{+}$ 
large enough such that, 
taking $e_0=t_{u} u$,  we have 
\(\left\|e_0\right\|\geqslant r_{0}\) and 
\(I\sb{\infty} (e_0) < 0\). 
This concludes the proof of item~\eqref{mpti2}.
\end{proof}

We denote by $d$ the mountain-pass level associated to 
the functional \( I\sb{\infty} \), that is, 
\begin{equation*}
d \equiv \inf\sb{\gamma \in \Gamma} \,
\max\sb{t \in [0,1]} I\sb{\infty}(\gamma(t)),
\end{equation*}
where
\begin{equation*}
\Gamma \equiv
\left\{ \gamma \in C([0,1]; 
D\sp{1,p}(B\sb{1}(0)) \cap 
D\sp{1,q}(B\sb{1}(0)))\colon
\gamma(0)=0 \text{ and } \gamma(1) = e_0 \right\}
\end{equation*}
and the function
\( e_0 
\in [D\sp{1,p}(B\sb{1}(0))\cap 
     D\sp{1,q}(B\sb{1}(0))] 
\backslash \{ 0 \} \) 
is given in Lemma~\ref{lem:I}.
It is standard to verify that the mountain-pass level 
\( d \) depends only 
on \( a\sb{\infty} \),  
on \( b\sb{\infty} \), 
on \( \theta \), and 
on the function \( f \).

For \( R >1 \) and for 
\( \theta > p \) given in hypothesis~\eqref{hyp:f3},
we set \( k \equiv \theta p/(\theta - p) > p \)
and we define a new nonlinearity 
\( g \colon \mathbb{R}\sp{N}\times \mathbb{R} 
\to \mathbb{R}\) by
\begin{equation*}
g(x,t) \equiv
\begin{cases}
f(t), & \text{if $|x| \leqslant R$ 
or if $|x| > R$ 
and 
$f(t) \leqslant \dfrac{a(x)}{k} \, |t|\sp{p-2}t$}; \\
\dfrac{a(x)}{k}\,|t|\sp{p-2}t,
& \text{if $|x| > R$ 
and $f(t) > \dfrac{a(x)}{k}\, |t|\sp{p-2}t$}.
\end{cases}
\end{equation*}
Using the notation
\( G(x,t) \equiv \displaystyle
\int\sb{0}\sp{t} g(x,s)\,\mathrm{d}s \), 
by direct computations we get the set of inequalities
\begin{alignat}{2}
g(x,t) & \leqslant \dfrac{a(x)}{k}\,|t|\sp{p-2}t,
&& \quad \text{for all $|x| \geqslant R$}; 
\label{eq:gxtf} \\
G(x,t) & = F(t),
&&\quad \text{if $|x| \leqslant R$}; 
\label{eq:capgf} \\
G(x,t) & \leqslant
\dfrac{a(x)}{kp} \, |t|\sp{p-1}t,
&& \quad \text{if $|x| > R>1$} 
\label{eq:capga}.
\end{alignat}

Now we define the auxiliary problem
\begin{equation}
\label{eq:auxprob}
\begin{cases}
-\Delta\sb{p}u -\Delta\sb{q}u 
+ a(x) \left|u\right|\sp{p-2} u 
+ b(x) \left|u\right|\sp{q-2} u 
= g(x,u), & \qquad x \in \mathbb{R}\sp{N};\\
\phantom{-}u(x) > 0, 
\quad u \in 
D\sp{1,p}(\mathbb{R}\sp{N})\cap 
D\sp{1,q}(\mathbb{R}\sp{N}), 
& \qquad x \in \mathbb{R}\sp{N}.
\end{cases}
\end{equation}
The Euler-Lagrange energy functional 
\( J \colon E \to \mathbb{R} \)
associated to the auxiliary problem~\eqref{eq:auxprob}
is given by
\begin{align*}
J(u) & \equiv 
\dfrac{1}{p}\int\sb{\mathbb{R}\sp{N}} 
|\nabla u|\sp{p}\,\mathrm{d}x
+ \dfrac{1}{p}\int\sb{\mathbb{R}\sp{N}} 
a(x)|u|\sp{p}\,\mathrm{d}x\\
& \qquad + 
\dfrac{1}{q}\int\sb{\mathbb{R}\sp{N}} 
|\nabla u|\sp{q}\,\mathrm{d}x
+ \dfrac{1}{q}\int\sb{\mathbb{R}\sp{N}} 
b(x)|u|\sp{q}\,\mathrm{d}x
- \int\sb{\mathbb{R}\sp{N}} G(x,u)\,\mathrm{d}x.
\end{align*}

Using the hypotheses on the nonlinearity
\( f \) and on the potential functions 
\( a \) and \( b \) we can show that 
\( J \in C\sp{1}(E; \mathbb{R}) \);
moreover, for every
\( u, v \in E \) 
its G\^{a}teaux derivative can be computed by
\begin{align*}
J'(u)v &= 
\int\sb{\mathbb{R}\sp{N}} 
|\nabla u|\sp{p-2} \nabla u 
\cdot 
\nabla v \,\mathrm{d}x
+ \int\sb{\mathbb{R}\sp{N}} 
a(x) |u|\sp{p-2}uv\,\mathrm{d}x \\
&\qquad + \int\sb{\mathbb{R}\sp{N}} 
|\nabla u|\sp{q-2} \nabla u 
\cdot 
\nabla v \,\mathrm{d}x
+ \int\sb{\mathbb{R}\sp{N}} 
b(x) |u|\sp{q-2}uv\,\mathrm{d}x
- \int\sb{\mathbb{R}\sp{N}} g(x,u) v\,\mathrm{d}x.
\end{align*}
As before, critical points of the energy functional 
\( J \) are weak solutions to 
problem~\eqref{eq:auxprob}. 

Our next goal is to apply the mountain-pass lemma to 
show that problem~\eqref{eq:auxprob} has a positive 
solution.

\begin{lem}
\label{lem:J}
The functional $J$  verifies the 
geometry of the mountain-pass lemma.
More precisely, the following claims are valid.
\begin{enumerate}
\item \label{mptj1}
There exist $r_1, \; \mu _1 \in \mathbb{R}\sp{+}$ 
such that  
$J(u)\geqslant \mu_1$ for $\left\|u\right\|=r_1$.
\item \label{mptj2}
There exists  
\(e_{1} \in 
\left[ D\sp{1,p}(B\sb{1}(0)) \cap 
 D\sp{1,q}(B\sb{1}(0)) \right] 
 \backslash \{ 0 \} \) 
such that  \(\left\|e_{1}\right\| \geqslant r_{1}\) 
and \(J(e_{1})<0\).
\end{enumerate}
\end{lem}
\begin{proof}
Using the equality~\eqref{eq:capgf} 
and the inequality~\eqref{eq:capga} together with the 
hypotheses~\eqref{hyp:f1} and~\eqref{hyp:f3} and the 
first inequality in~\eqref{eq:conshypf1f2}, we obtain 
\begin{align*}
J(u) & \geqslant 
\dfrac{1}{p}\left\|u\right\|_{1,p}\sp{p}+ \dfrac{1}{q}\left\|u\right\|_{1,q}\sp{q} 
- \int\sb{|x| \leqslant R} 
F(u)\,\mathrm{d}x 
- \int\sb{|x| > R} 
\dfrac{a(x)\left|u\right|\sp{p}}{kp}\,\mathrm{d}x\\
& \geqslant 
\dfrac{1}{p}\left\|u\right\|_{1,p}\sp{p}
+ \dfrac{1}{q}\left\|u\right\|_{1,q}\sp{q}
- \dfrac{c_{0}}{\theta}\int\sb{\mathbb{R}\sp{N}} 
\left|u\right|\sp{p\sp{*}}\,\mathrm{d}x
- \dfrac{1}{kp}\left\|u\right\|_{1,p}\sp{p}\\
& =
\left(\dfrac{1}{p}- \dfrac{1}{kp}\right)
\left\|u\right\|_{1,p}\sp{p} 
+ \dfrac{1}{q}\left\|u\right\|_{1,q}\sp{q}
- \dfrac{c_{0}}{\theta} \left|u\right|_{L\sp{p\sp{*}}}\sp{p\sp{*}}.
\end{align*}

Now we apply the Sobolev inequality
\begin{equation}
\label{eq:Sob}
\| u \|_{L\sp{m\sp{*}} (\mathbb{R}\sp{N})}\sp{m}
\leqslant S\sb{m}
\int\sb{\mathbb{R}\sp{N}} 
\left|\nabla u\right|\sp{m}\,\mathrm{d}x 
\quad \textrm{for all\,}
u \in D\sp{1,m}(\mathbb{R}\sp{N}) 
\qquad (m \in \{ p,q \})
\end{equation}
in the computations above and set 
\( S \equiv \max \{ S\sb{p}, S\sb{q} \} \) to get
\begin{align*}
J(u)&\geqslant 
\left(\dfrac{1}{p}- \dfrac{1}{kp}\right)
\left\|u\right\|_{1,p}\sp{p} 
+ \dfrac{1}{q}\left\|u\right\|_{1,q}\sp{q}
- \dfrac{c_{0}}{\theta} S\sp{p\sp{*}/p}
\left(\int\sb{\mathbb{R}\sp{N}} 
\left|\nabla u\right|\sp{p}\,\mathrm{d}x 
\right)\sp{p\sp{*}/p}\\
&\geqslant
\min\left\{\dfrac{1}{p}- \dfrac{1}{kp}, \dfrac{1}{q}\right\}\left(\left\|u\right\|_{1,p}\sp{p}+ \left\|u\right\|_{1,q}\sp{q}\right)
- \dfrac{c_{0}}{\theta} S\sp{p\sp{*}/p}
\left(\left\|u\right\|_{1,p}\sp{p}+ \left\|u\right\|_{1,q}\sp{q} \right)\sp{p\sp{*}/p}.
\end{align*}
If we take $\left\|u\right\|_{1,p}$  and 
$\left\|u\right\|_{1,q}$ small enough, 
it follows that  
$\left\|u\right\|_{1,p}\sp{p}$ and 
$\left\|u\right\|_{1,q}\sp{q}$ are also small enough. 
For that reason, we obtain
the existence of $r_{1}$, 
$\mu_{1} \in \mathbb{R}\sp{+}$ such that 
$J(u)\geqslant \mu_{1}$ for 
\( \| u \| = r\sb{1}\). 
This concludes the proof of item~\eqref{mptj1}. 

By  definition we have  that $G(x,u)=F(u)$ for  all 
\( u \in 
\left[ D\sp{1,p}(B\sb{1}(0)) \cap 
D\sp{1,q}(B\sb{1}(0)) \right] 
\backslash \{ 0 \} \).  
Arguing as in the proof of Lemma~\ref{lem:I} 
we conclude  that there exist  
$r_{1}, t\sb{u} \in \mathbb{R}\sp{+}$ such that 
$e_{1}\equiv t_{u}u$ verify the inequalities 
\( \left\| e_{1} \right\| \leqslant r\sb{1} \)
and 
$J(e_{1})<0$. 
This concludes the proof of item~\eqref{mptj2}. 
The lemma is proved.
\end{proof}

Since the functional \( J \) has the 
geometry of the mountain-pass lemma, 
using Willem~\cite[Theorem~1.15]{bib:w} we obtain a 
Palais-Smale sequence 
\( (u\sb{n})\sb{n\in \mathbb{N}} \subset E \) 
such that
\( J(u\sb{n}) \to c \)
and \( J'(u\sb{n}) \to 0\)
as \( n \to +\infty \). 
Here \( c \in \mathbb{R}\sp{+} \) 
is the mountain-pass level associated to the 
energy functional \( J \), that is,
\begin{equation*}
c \equiv \inf\sb{\gamma \in \Gamma} \,
\max\sb{t \in [0,1]} J(\gamma(t)),
\end{equation*}
where
\begin{equation*}
\Gamma \equiv
\left\{\gamma \in C([0,1]; 
D\sp{1,p}(B\sb{1}(0)) \cap 
D\sp{1,q}(B\sb{1}(0)) \colon
\gamma(0)=0 \text{ and } \gamma(1) = e_{1} \right\}
\end{equation*}
and 
\( e_{1} \in 
\left[ D\sp{1,p}(B\sb{1}(0)) 
\cap D\sp{1,q}(B\sb{1}(0)) \right]
\backslash \{ 0 \} \) 
is the same function verifying inequality 
\( J(e_{1}) < 0 \) in Lema~\ref{lem:J}.
Using the hypothesis~\eqref{hyp:f4}, 
without loss of generality 
we can suppose that the sequence
\( (u\sb{n})\sb{n \in \mathbb{N}} \subset E \)
consists of nonnegative functions.

We note that for all 
\( u \in \left[D\sp{1,p}(B\sb{1}(0))
    \cap  D\sp{1,q}(B\sb{1}(0))\right] 
    \backslash \{ 0 \} \) 
the inequality
\( J(u) \leqslant I\sb{\infty}(u) \)
is valid, and this implies that 
\begin{equation} \label{c:d}
c \leqslant d .
\end{equation}

Now we prove the boundedness of the Palais-Smale 
sequences for the functional \( J\).

\begin{lem}
\label{lem:psbounded}
Suppose that the potential functions 
\( a \), \( b \) verify the hypothesis~\eqref{hyp:v1}, 
and that the nonlinearity \( f \) verifies the  
hypotheses~\eqref{hyp:f1},~\eqref{hyp:f2},~\eqref{hyp:f3}, 
and~\eqref{hyp:f4}.  
If
\( (u\sb{n})\sb{n \in \mathbb{N}} \subset E \)
is a Palais-Smale sequence for the energy functional 
\( J \), then the sequence
\( (u\sb{n})\sb{n \in \mathbb{N}} \subset E \)
is bounded in \( E \).
\end{lem}

\begin{proof}   
To obtain our thesis it is sufficient 
to prove that both sequences
\( (\left\| u_n \right\|\sb{1,q}\sp{q})
\sb{n \in \mathbb{N}} 
\subset \mathbb{R} \)
and
\( (\left\| u_n \right\|\sb{1,p}\sp{p})
\sb{n \in \mathbb{N}} 
\subset \mathbb{R} \)
are bounded, which we do in the two claims below. 

Before that, however, we remark that 
there  exist constants \(c_1 > 0 \) and 
\(n_0 \in \mathbb{N}\) such that 
\( J(u\sb{n}) \leqslant c_1 \) and 
\( \left| J'(u\sb{n}u\sb{n})\right|
\leqslant 
\min \big\{ \left\|u\sb{n}\right\|\sb{1,q},  
\left\|u\sb{n}\right\|\sb{1,p} \big\} \) 
for all \( n \in \mathbb{N} \) such that 
$n \geqslant n_0$; 
and since $\theta>p>1$, for all $n \geqslant n_0$ 
we have 
\begin{equation} \label{J:sup}
J(u\sb{n})  
- \frac{1}{\theta}J'(u\sb{n})u\sb{n} 
\leqslant c_1  
+ \frac{1}{\theta } 
\left\|u\sb{n}\right\|  
\leqslant c_1  + 
\min \big\{ \left\|u\sb{n}\right\|\sb{1,q}, 
\left\|u\sb{n}\right\|\sb{1,p} \big\}.
\end{equation}

\begin{claim}
\label{hyp:ff1}
The sequence 
\( (\left\| u_n \right\|\sb{1,q}\sp{q})
\sb{n \in \mathbb{N}} 
\subset \mathbb{R} \) is bounded.
\end{claim}
\begin{proof}[Proof of Claim~\ref{hyp:ff1}]
We divide our analysis into cases that mirror the 
definition of the nolinearity \( g \).
If
\( |x| > R \) and 
\( f(t) > a(x)|t|\sp{p-2}t/k\), 
then 
\begin{align*}
\int\sb{\mathbb{R}\sp{N}} G(x,u_n)\,\mathrm{d}x 
& = \frac{1}{p}\int\sb{\mathbb{R}\sp{N}} 
g(x,u_n)u_n\,\mathrm{d}x,
\end{align*}
and this implies that 
\begin{align}
J(u\sb{n})  - \frac{1}{p}J'(u\sb{n})u\sb{n} 
& = 
\left(\dfrac{1}{q}-\dfrac{1}{p}\right)
\left\| u_n \right\|\sb{1, q}\sp{q}.
\label{eq:designnormq}
\end{align}
Combining inequalities~\eqref{J:sup} 
and~\eqref{eq:designnormq} we conclude that
\begin{align*} 
\left(\dfrac{1}{q}-\dfrac{1}{p}\right)
\left\| u_n\right\|\sb{1, q}\sp{q}
& \leqslant  c\sb{1} 
+ \left\| u_n\right\|\sb{1, q}.
\end{align*}
So, in this case the sequence 
\( (\left\| u_n\right\|\sb{1, q}\sp{q} )
\sb{n \in \mathbb{N}}
\subset \mathbb{R}\) 
is bounded, 
say \( \left\| u_n\right\|\sb{1, q}\sp{q}
\leqslant c\sb{q} \) for every \( n \in \mathbb{N}\).

If \(|x| \leqslant R\)
or if
\( |x| > R \) and  
\( f(t) \leqslant a(x)|t|\sp{p-2}t/k \),
the boundedness of the sequence can be proved using the 
same ideas as that of the previous case with some minor 
changes.
This concludes the proof of the claim.
\end{proof}

\begin{claim}
\label{hyp:ff2}
The sequence 
\( (\left\| u_n\right\|\sb{1,p}\sp{p})
\sb{n \in \mathbb{N}} 
\subset \mathbb{R} \) is bounded.
\end{claim}
\begin{proof}[Proof of Claim~\ref{hyp:ff2}]
We also divide our analysis into the same cases.
If \( |x| > R \) and 
\( f(t) > a(x)|t|\sp{p-2}t/k \), 
then we have
\begin{align}
J(u\sb{n})  - \frac{1}{\theta}J'(u\sb{n})u\sb{n} 
&\geqslant \left(\dfrac{1}{p}-\dfrac{1}{\theta}\right)
\left\|u_n\right\|_{1,p}\sp{p}
+\left( \dfrac{1}{q} -\dfrac{1}{\theta}\right)
\left\|u_n\right\|_{1,q}\sp{q}
- \dfrac{1}{kp} 
\left\{\int\sb{\mathbb{R}\sp{N}} 
a(x)|u_n|\sp{p}\,\mathrm{d}x \right\} \nonumber \\
& \geqslant \left(\dfrac{1}{p}
- \dfrac{1}{\theta}\right)\left\|u_n\right\|_{1,p}\sp{p} 
+ \left(\dfrac{1}{p}
- \dfrac{1}{\theta}\right)\left\|u_n\right\|_{1,p}\sp{q} 
- \dfrac{1}{kp} 
\left\{\left\|u_n\right\|_{1,p}\sp{p} 
+ \left\|u_n\right\|_{1,q}\sp{q} \right\} \nonumber \\
& = \frac{(p-1)}{kp}
\left\{\left\|u_n\right\|_{1,p}\sp{p} 
+ \left\|u_n\right\|_{1,q}\sp{q} \right\}.
\label{J:inf}
\end{align} 
Combining inequalities~\eqref{J:sup} 
and~\eqref{J:inf} and using
Claim~\ref{hyp:ff1} we obtain
\begin{align*}
\frac{(p-1)}{kp}\left\|u_n\right\|_{1,p}\sp{p}  
& \leqslant c_1 + \left\|u_n\right\|_{1,p}.
\end{align*}
This means that in this case 
the sequence 
\( (\left\|u_n\right\|_{1,p}\sp{p})
\sb{n \in \mathbb{N}} \subset \mathbb{R} \) 
is bounded. 

If \(|x| \leqslant R\) or if 
\( |x| > R \) and  
\( f(t) \leqslant a(x)|t|\sp{p-2}t/k \), 
then 
\begin{align*}
\int\sb{\mathbb{R}\sp{N}} G(x,u_n)\,\mathrm{d}x 
+ \dfrac{1}{\theta}\int\sb{\mathbb{R}\sp{N}} 
g(x,u_n)u_n\,\mathrm{d}x \geqslant 0.
\end{align*}
Hence,
\begin{align}
{} & J(u\sb{n}) - \frac{1}{\theta}J'(u\sb{n})u\sb{n}  
\nonumber \\
& \qquad \geqslant 
\left(\dfrac{1}{p}-\dfrac{1}{\theta}\right)
\left\|u_n\right\|_{1,p}\sp{p} 
+ \left( \dfrac{1}{q} -\dfrac{1}{\theta}\right)
\left\|u_n\right\|_{1,q}\sp{q}
- \int\sb{\mathbb{R}\sp{N}} 
G(x,u_n)\,\mathrm{d}x 
+ \frac{1}{\theta}\int\sb{\mathbb{R}\sp{N}} 
g(x,u_n)u_n\,\mathrm{d}x \nonumber \\
& \qquad \geqslant 
\left(\dfrac{1}{p}-\dfrac{1}{\theta} \right)
\left\{\left\|u_n\right\|_{1,p}\sp{p} 
+\left\|u_n\right\|_{1,q}\sp{q} \right\}
- \int\sb{\mathbb{R}\sp{N}} G(x,u_n)\,\mathrm{d}x 
+ \frac{1}{\theta}\int\sb{\mathbb{R}\sp{N}} 
g(x,u_n)u_n\,\mathrm{d}x \nonumber \\
& \qquad 
\geqslant \dfrac{1}{k}
\left\{\left\|u_n\right\|_{1,p}\sp{p} 
+ \left\|u_n\right\|_{1,q}\sp{q} \right\}
\geqslant \dfrac{(p-1)}{kp}
\left\{\left\|u_n\right\|_{1,p}\sp{p} 
+ \left\|u_n\right\|_{1,q}\sp{q} \right\}.
\label{J:infk}
\end{align}
Combining inequalities~\eqref{J:sup} 
and~\eqref{J:infk} we get
\begin{align*}
\frac{1}{k}\left\|u_n\right\|_{1,p}\sp{p}  
& \leqslant c_1 +\left\|u_n\right\|_{1,p}.
\end{align*}
This means that also in this case 
the sequence
\( (\left\|u_n\right\|_{1,p}\sp{p})
\sb{n \in \mathbb{N}} \subset \mathbb{R} \) 
is bounded. 
This concludes the proof of the claim.
\end{proof}
Using Claims~\ref{hyp:ff1} and~\ref{hyp:ff2} we deduce 
the proof of the lemma.
\end{proof}

The following result shows that the functional \( J \) 
verifies the Palais-Smale condition.

\begin{lem}
\label{lem:funcjps}
Suppose that the potential functions \( a \), \( b \) 
verify the hypotheses~\eqref{hyp:v1},~\eqref{hyp:v2}, 
and~\eqref{hyp:v3} 
and that the nonlinearity  \( f \) verifies the 
hypotheses~\eqref{hyp:f1},~\eqref{hyp:f2},~\eqref{hyp:f3}, 
and~\eqref{hyp:f4}.
Then the Palais-Smale condition is valid for the 
energy functional \( J \).
\end{lem}
\begin{proof}
Let  $(u_n)\sb{n\in N}\subset E $  be a 
Palais-Smale  sequence at the level $c$; 
this means that 
\begin{align*}
J(u_n)\rightarrow c\mbox{ and }
J'(u_n)\rightarrow 0
\end{align*}
as $n\rightarrow\infty$.
By  Lema~\ref{lem:psbounded} this sequence is bounded. 
Then there exist a subsequence of 
$(u_n)\sb{n\in N}\subset E $,  which we still denote in 
the same way, and there exists a function \( u \in E\) 
such that 
\(u_n \rightharpoonup u\) weakly in \( E \) 
as \( n \to +\infty\).

For each $\epsilon >0$, there exist  $r>R>1$ such that
\begin{align}
\label{d:eps}
2(2\sp{N}-1)\sp{1/N} 
\omega\sp{\frac{1}{N}}_N 
\left(1-\frac{1}{k}\right)\sp{-1} 
& 
\left\{ \splitfrac{\left( 
\displaystyle\int\sb{r\leqslant \left|x\right|
\leqslant 2r}\left|u\right|\sp{p\sp{*}} 
\mathrm{d}x\right)\sp{1/p\sp{*}}
\left\|u\right\|\sp{p-1} }
{ + \left( \displaystyle\int\sb{r\leqslant 
\left| x \right| \leqslant 2r}
\left|u\right|\sp{q\sp{*}} 
\mathrm{d}x\right)\sp{1/q\sp{*}}
\left\|u\right\|\sp{q-1} }
\right\} 
< \epsilon.
\end{align}

Let  
\( \eta = \eta_r \in C\sp{\infty}(B\sb{r}\sp{c}(0)) \) 
be a cut off function such that 
\(0\leqslant\eta\leqslant 1\),
with
$\eta =1$  in  
$B\sb{2r}\sp{c}(0)$ 
and also 
$\left|\nabla \eta\right|\leqslant 2/r$
for all $x\in\mathbb{R}\sp{N}$. 
Since the sequence 
$(u_n)\sb{n\in N}\subset E $ 
is bounded, it follows that the sequence 
$(\eta u_n)\sb{n\in N}\subset E $ is bounded also. 
Therefore, $J'(u_n)(\eta u_n)=o_n(1)$, that is, 
\begin{align}
\label{g:g}
& \int \sb{\mathbb{R}\sp{N}} 
|\nabla u_n|\sp{p-2} \nabla u_n \cdot 
\nabla (\eta u_n)\,\mathrm{d}x
+ \int\sb{\mathbb{R}\sp{N}} 
a(x) |u_n|\sp{p-2}u_n(\eta u_n)\,\mathrm{d}x\nonumber \\
& \qquad + \int\sb{\mathbb{R}\sp{N}} 
|\nabla u_n|\sp{q-2} \nabla u_n \cdot 
\nabla (\eta u_n) \,\mathrm{d}x
+ \int\sb{\mathbb{R}\sp{N}} 
b(x) |u|_n\sp{q-2}u_n(\eta u_n)\,\mathrm{d}x \nonumber\\
& \qquad \qquad = \int\sb{\mathbb{R}\sp{N}} 
g(x,u_n) (\eta u_n)\,\mathrm{d}x + o(1).
\end{align}
The previous expression and the properties of the cut off function $\eta$ 
imply that
\begin{align*}
{} & \int \sb{\left|x\right|\geqslant r} 
\eta |\nabla u_n|\sp{p}\,\mathrm{d}x 
+\int \sb{\left|x\right|\geqslant r} 
|\nabla u_n|\sp{p-2} u_n \nabla u_n \cdot 
\nabla \eta \,\mathrm{d}x
+ \int\sb{\left|x\right|\geqslant r} 
\eta a(x) |u_n|\sp{p}\,\mathrm{d}x \\
& \qquad 
+ \int \sb{\left|x\right|\geqslant r} 
\eta |\nabla u_n|\sp{q}\,\mathrm{d}x 
+\int \sb{\left|x\right|\geqslant r}
|\nabla u_n|\sp{q-2} u_n \nabla u_n \cdot 
\nabla \eta \,\mathrm{d}x
+ \int\sb{\left|x\right|\geqslant r} 
\eta b(x) |u_n|\sp{q}\,\mathrm{d}x \\
& \qquad \qquad 
= \int\sb{\left|x\right|\geqslant r}
\eta g(x,u_n)  u_n\,\mathrm{d}x + o(1).
\end{align*}
By the inequality~\eqref{eq:gxtf},
it follows that 
\begin{align*}
\int\sb{\left|x\right|\geqslant r}
\eta g(x,u_n)  u_n\,\mathrm{d}x 
& \leqslant
\int\sb{\left|x\right|\geqslant r} 
\eta \dfrac{a(x)}{k}\,|u_n|\sp{p} \, \mathrm{d}x;
\end{align*}
thus, we obtain
\begin{align*}
{} & \int \sb{\left|x\right|\geqslant r} 
\eta |\nabla u_n|\sp{p}\,\mathrm{d}x 
+ \int\sb{\left|x\right|\geqslant r} 
\eta a(x) |u_n|\sp{p}\,\mathrm{d}x \\
& \qquad + \int \sb{\left|x\right|\geqslant r}
\eta |\nabla u_n|\sp{q}\,\mathrm{d}x 
+ \int\sb{\left|x\right|\geqslant r} 
\eta b(x) |u_n|\sp{q}\,\mathrm{d}x
- \int\sb{\left|x\right|\geqslant r} 
\eta \dfrac{a(x)}{k}\,|u_n|\sp{p} \\
& \qquad \qquad \leqslant 
\int \sb{\left|x\right|\geqslant r} 
|\nabla u_n|\sp{p-1} |u_n| |\nabla \eta| \,\mathrm{d}x 
+ \int \sb{\left|x\right|\geqslant r}
|\nabla u_n|\sp{q-1} |u_n|
|\nabla \eta| \,\mathrm{d}x +o(1) \\
& \qquad \qquad \leqslant \dfrac{2}{r} 
\left\{\int \sb{r\leqslant\left|x\right|\leqslant 2r}
|\nabla u_n|\sp{p-1}|u_n| \,\mathrm{d}x 
+ \int \sb{r\leqslant\left|x\right|\leqslant 2r}
|\nabla u_n|\sp{q-1} |u_n|
\,\mathrm{d}x\right\} + o(1).
\end{align*}
Subtracting the terms  
\begin{align*} 
\dfrac{1}{k} 
\int \sb{\left|x\right|\geqslant r} 
\eta |\nabla u_n|\sp{p}\,\mathrm{d}x  
+ \dfrac{1}{k} 
\int \sb{\left|x\right|\geqslant r}
\eta |\nabla u_n|\sp{q}\,\mathrm{d}x 
+ \dfrac{1}{k}\int\sb{\left|x\right|\geqslant r} 
\eta b(x) |u_n|\sp{q}\,\mathrm{d}x
\end{align*}
from the left-hand side of the previous inequality 
and grouping the several integrals, we deduce that 
\begin{align*}
{} & \left(1-\dfrac{1}{k}\right)
\left\{ \splitfrac{\displaystyle
\int \sb{\left|x\right|\geqslant r}
\eta |\nabla u_n|\sp{p}\,\mathrm{d}x 
+ \displaystyle\int\sb{\left|x\right|\geqslant r} 
\eta a(x) |u_n|\sp{p}\,\mathrm{d}x }
{+ \displaystyle\int \sb{\left|x\right|\geqslant r}
\eta |\nabla u_n|\sp{q}\,\mathrm{d}x 
+ \displaystyle\int\sb{\left|x\right|\geqslant r} 
\eta b(x) |u_n|\sp{q}\,\mathrm{d}x }
\right\} \\
& \qquad \qquad \leqslant \dfrac{2}{r} 
\left\{ \int \sb{r\leqslant
\left|x\right|\leqslant 2r}
\left|u_n\right| |\nabla u_n|\sp{p-1} \,\mathrm{d}x 
+ \int \sb{r\leqslant\left|x\right|\leqslant 2r}
\left|u_n\right| |\nabla u_n|\sp{q-1}
\,\mathrm{d}x 
\right\} +o(1).
\end{align*}

Now we use H\"{o}lder's inequality to get
\begin{align*}
\int \sb{r\leqslant\left|x\right|\leqslant 2r}
\left|u_n\right| |\nabla u_n|\sp{p-1} \,\mathrm{d}x 
& \leqslant 
\left(\int \sb{r\leqslant\left|x\right|\leqslant 2r}
|u_n|\sp{p} \,\mathrm{d}x\right)\sp{1/p}
\left\{\left(
\int \sb{r\leqslant\left|x\right|\leqslant 2r}
|\nabla u_n|\sp{p} 
\,\mathrm{d}x \right)
\sp{1/p}\right\}\sp{p-1}\\
& \leqslant 
\left(\int \sb{r\leqslant\left|x\right|\leqslant 2r}
|u_n|\sp{p} \,\mathrm{d}x\right)\sp{1/p}
\left\|u_n\right\|\sp{p-1}.
\end{align*}
And in a similar way, we obtain
\begin{align*}
\int \sb{r\leqslant\left|x\right|\leqslant 2r}
\left|u_n\right| |\nabla u_n|\sp{q-1} \,\mathrm{d}x 
& \leqslant 
\left(\int \sb{r\leqslant \left|x\right|\leqslant 2r}
|u_n|\sp{q} \,\mathrm{d}x\right)\sp{\frac{1}{q}}
\left\|u_n\right\|\sp{q-1}.
\end{align*}
By the compactness of the embedding 
$W\sp{1,p}(\overline{B}\sb{2r} \backslash B\sb{r}) 
\hookrightarrow L\sp{p}
(\overline{B}\sb{2r} \backslash B\sb{r})$,
we infer that $u_n \to u$ strongly in
$L\sp{p}(\overline{B}\sb{2r} \backslash B\sb{r})$ 
as $n \to \infty$. 
Since $(\eta u_n)\sb{n\in\mathbb{N}} 
\subset  W\sp{1,p}(\mathbb{R}\sp{N}) \cap 
         W\sp{1,q}(\mathbb{R}\sp{N})$, 
it follows that
\begin{align}
\label{lim:sup}
{} & \limsup\sb{n\rightarrow\infty}
\left(1-\dfrac{1}{k}\right) 
\left\{ \splitfrac{\displaystyle
\int\sb{\left|x\right|\geqslant r}
\eta |\nabla u_n|\sp{p}\,\mathrm{d}x 
+ \displaystyle\int\sb{\left|x\right|\geqslant r} 
\eta a(x) |u_n|\sp{p}\,\mathrm{d}x }
{+ \displaystyle\int\sb{\left|x\right|\geqslant r}
\eta |\nabla u_n|\sp{q}\,\mathrm{d}x 
+ \displaystyle\int\sb{\left|x\right|\geqslant r} 
\eta b(x) |u_n|\sp{q}\,\mathrm{d}x } \right\} 
\nonumber\\
& \qquad  
\leqslant  \dfrac{2}{r} \limsup\sb{n\rightarrow\infty}
\left\{ 
\left(\int \sb{r\leqslant \left|x\right|\leqslant 2r}
|u_n|\sp{p} \,\mathrm{d}x\right)\sp{1/p}
\left\|u_n\right\|\sp{p-1}
+ \left(\int \sb{r\leqslant\left|x\right|\leqslant 2r}
|u_n|\sp{q} \,\mathrm{d}x\right)\sp{1/q}
\left\|u_n\right\|\sp{q-1}\right\} \nonumber\\
& \qquad 
=\dfrac{2}{r}
\left\{ 
\left(\int \sb{r\leqslant\left|x\right|\leqslant 2r}
|u|\sp{p} \,\mathrm{d}x\right)\sp{1/p}
\left\|u\right\|\sp{p-1}
+ \left(\int \sb{r\leqslant\left|x\right|\leqslant 2r}
|u|\sp{q} \,\mathrm{d}x\right)\sp{1/q}
\left\|u\right\|\sp{q-1}\right\}.
\end{align}

Applying H\"{o}lder's inequality once more and 
denoting the volume of the unitary ball by 
\( | B\sb{1}(0) | = \omega\sb{N} \), 
we obtain
\begin{align}
\label{cota:p}
\left( \int \sb{r\leqslant \left|x\right|\leqslant 2r}
\left|u\right|\sp{p} \,\mathrm{d}x\right) \sp{1/p}  
& \leqslant 
\big((2\sp{N}-1)\omega \sb{_N}r\sp{N}\big)\sp{1/N}
\left(\int \sb{r\leqslant\left|x\right|\leqslant 2r}
|u|\sp{p\sp{*}}\,\mathrm{d}x\right)\sp{1/p\sp{*}}.
\end{align}
And in a similar way, we obtain
\begin{align}
\label{cota:q}
\left( \int \sb{r\leqslant\left|x\right|\leqslant 2r}
\left|u\right|\sp{q} \,\mathrm{d}x\right) \sp{1/q}  
& \leqslant 
\big((2\sp{N}-1)\omega \sb{_N}r\sp{N}\big)\sp{1/N}
\left(\int \sb{r\leqslant\left|x\right|\leqslant 2r}
|u|\sp{q*}\,\mathrm{d}x\right)\sp{1/q\sp{*}}.
\end{align}
Substituting inequalities~\eqref{cota:p} 
and~\eqref{cota:q} 
in~\eqref{lim:sup}, we get
\begin{align}
\label{lim:supf}
{} & 
\limsup\sb{n\rightarrow\infty}
\left(1-\dfrac{1}{k}\right) 
\left\{ \splitfrac{\displaystyle
\int \sb{\left|x\right|\geqslant r}
\eta |\nabla u_n|\sp{p}\,\mathrm{d}x 
+ \displaystyle\int\sb{\left|x\right|\geqslant r} 
\eta a(x) |u_n|\sp{p}\,\mathrm{d}x }
{+ \displaystyle\int \sb{\left|x\right|\geqslant r}
\eta |\nabla u_n|\sp{q}\,\mathrm{d}x 
+ \int\sb{\left|x\right|\geqslant r} 
\eta b(x) |u_n|\sp{q}\,\mathrm{d}x }
\right\} \nonumber\\
& \qquad 
\leqslant 2 \left((2\sp{N}-1) 
\omega\sb{_N}\right)\sp{1/N}
\left\{
\splitfrac{\left(\displaystyle
\int \sb{r\leqslant\left|x\right|\leqslant 2r}
|u|\sp{p*}\,\mathrm{d}x\right)\sp{1/p\sp{*}}
\left\|u\right\|\sp{p-1} }
{+\left(\displaystyle
\int \sb{r\leqslant\left|x\right|\leqslant 2r}
|u|\sp{q*}\,\mathrm{d}x\right)\sp{1/q\sp{*}}
\left\|u\right\|\sp{q-1} }
\right\}.
\end{align}
In particular, since $\eta =1$  
outside the ball of radius $2r$, by 
inequalities~\eqref{lim:sup} and~\eqref{lim:supf} 
we obtain
\begin{align}
{} & 
\limsup\sb{n\rightarrow\infty}
\left(1-\dfrac{1}{k}\right) 
\left\{ \splitfrac{\displaystyle
\int \sb{\left|x\right|\geqslant 2r} 
|\nabla u_n|\sp{p}\,\mathrm{d}x 
+ \displaystyle\int\sb{\left|x\right|\geqslant 2r} 
a(x) |u_n|\sp{p}\,\mathrm{d}x }
{+ \displaystyle\int \sb{\left|x\right|\geqslant 2r}
|\nabla u_n|\sp{q}\,\mathrm{d}x 
+ \int\sb{\left|x\right|\geqslant 2r}  
b(x) |u_n|\sp{q}\,\mathrm{d}x }
\right\} \nonumber \\
& \qquad \leqslant 
2\left((2\sp{N}-1) \omega \sb{_N}\right)\sp{1/N}
\left\{ \splitfrac{ \left(
\displaystyle
\int \sb{r\leqslant \left|x\right|\leqslant 2r}
|u|\sp{p\sp{*}}\,\mathrm{d}x\right)\sp{1/p\sp{*}}
\left\|u\right\|\sp{p-1} }
{ +\left( \displaystyle
\int \sb{r\leqslant \left|x\right|\leqslant 2r}
|u|\sp{q\sp{*}}\,\mathrm{d}x\right)\sp{1/q\sp{*}}
\left\|u\right\|\sp{q-1} }
\right\}.
\label{lim:supf2r}
\end{align}
Therefore, by inequalities~\eqref{d:eps} 
and~\eqref{lim:supf2r} it follows that
\begin{align}
\label{d:final}
\limsup\sb{n\rightarrow\infty}\left\{
\splitfrac{\displaystyle
\int \sb{\left|x\right|\geqslant 2r} 
|\nabla u_n|\sp{p}\,\mathrm{d}x 
+\displaystyle\int\sb{\left|x\right|\geqslant 2r} 
a(x) |u_n|\sp{p}\,\mathrm{d}x }
{+ \displaystyle \int \sb{\left|x\right|\geqslant 2r} 
|\nabla u_n|\sp{q}\,\mathrm{d}x 
+ \displaystyle\int\sb{\left|x\right|\geqslant 2r}  
b(x) |u_n|\sp{q}\,\mathrm{d}x }
\right\} 
< \epsilon.
\end{align}
Combining inequalities~\eqref{g:g} and~\eqref{d:final}, 
we deduce that 
\begin{align}
\label{lim:g}
\limsup\sb{n\rightarrow\infty}
\int \sb{\left|x\right|\geqslant 2r} 
g(x,u_n)u_n\,\mathrm{d}x 
& = 0.
\end{align}

Now we use the dominated convergence theorem together 
with the fact that $g$ has subcritical growth 
to infer that
\begin{align}
\label{lim:gg}
\limsup\sb{n\rightarrow\infty} 
\int \sb{\left|x\right|\leqslant 2r} g(x,u_n)u_n\,
\mathrm{d}x 
& = \int \sb{\left|x\right|\leqslant 2r} 
g(x,u)u\,\mathrm{d}x;
\end{align}
and since 
$\displaystyle 
\int \sb{\mathbb{R}\sp{N}} 
g(x,u_n)u_n\,\mathrm{d}x<\infty$,  
by the choice of $r > R > 1$ and from  
equalities~\eqref{lim:g} and~\eqref{lim:gg}, we obtain
\begin{align}
\label{lim:ggg}
\lim\sb{n \to \infty}\int \sb{\mathbb{R}\sp{N}} 
g(x,u_n)u_n\,\mathrm{d}x 
& = \int \sb{\mathbb{R}\sp{N}} 
g(x,u)u\,\mathrm{d}x.
\end{align}

It remains to show that the norm sequence
\( (\left\|u_n\right\|)\sb{n \in \mathbb{N}} 
\subset \mathbb{R} \)
is such that 
\( \left\| u\sb{n} \right\| 
\to \left\|u\right\| \in \mathbb{R} \)
as \( n \to \infty \).  
Using H\"{o}lder's inequality and 
making some computations, it follows that
\begin{align*}
o(1) & = \left(J'(u_n) - J'(u)\right)(u_n - u) \\
&  \geqslant 
\left\{ \splitfrac{
\left(\displaystyle\int\sb{\mathbb{R}\sp{N}} 
|\nabla u_n|\sp{p}   \,\mathrm{d}x \right)\sp{(p-1)/p} }
{- \left(\displaystyle\int\sb{\mathbb{R}\sp{N}} 
|\nabla u|\sp{p}
\,\mathrm{d}x \right)\sp{(p-1)/p}}
\right\} 
\times 
\left\{ \splitfrac{
\left(\displaystyle\int\sb{\mathbb{R}\sp{N}} 
|\nabla u_n|\sp{p}\,\mathrm{d}x \right)\sp{1/p} }
{- \left(\displaystyle
\int\sb{\mathbb{R}\sp{N}} |\nabla u|\sp{p} 
\,\mathrm{d}x \right)\sp{1/p}}
\right\} \\
& \qquad 
+ \left\{ \splitfrac{
\left(\displaystyle\int\sb{\mathbb{R}\sp{N}} 
a(x)| u_n|\sp{p} \,\mathrm{d}x \right)\sp{(p-1)/p} }
{- \left(\displaystyle\int\sb{\mathbb{R}\sp{N}} 
a(x)|u|\sp{p} 
\,\mathrm{d}x \right)\sp{(p-1)/p} }
\right\}  
\times 
\left\{ \splitfrac{
\left(\displaystyle\int\sb{\mathbb{R}\sp{N}} 
a(x)|u_n|\sp{p} \,\mathrm{d}x \right)\sp{1/p} }
{- \left(\displaystyle\int\sb{\mathbb{R}\sp{N}} 
a(x)|u|\sp{p} \,\mathrm{d}x \right)\sp{1/p} }
\right\}\\
& \qquad  
+ \left\{ \splitfrac{
\left(\displaystyle\int\sb{\mathbb{R}\sp{N}} 
|\nabla u_n|\sp{q} \,\mathrm{d}x \right)\sp{(q-1)/q} }
{- \left(\displaystyle\int\sb{\mathbb{R}\sp{N}} 
|\nabla u|\sp{q} 
\,\mathrm{d}x \right)\sp{(q-1)/q} }
\right\} 
\times 
\left\{ \splitfrac{
\left(\displaystyle\int\sb{\mathbb{R}\sp{N}} 
|\nabla u_n|\sp{q} \,\mathrm{d}x\right)\sp{1/q} }
{- \left(\displaystyle\int\sb{\mathbb{R}\sp{N}} 
|\nabla u|\sp{q} 
\,\mathrm{d}x \right)\sp{1/q} }
\right\}\\
& \qquad
+ \left\{ \splitfrac{
\left(\displaystyle\int\sb{\mathbb{R}\sp{N}} 
b(x)| u_n|\sp{q} \,\mathrm{d}x \right)\sp{(q-1)/q} }
{- \left(\displaystyle\int\sb{\mathbb{R}\sp{N}} 
b(x)|u|\sp{q}
 \,\mathrm{d}x \right)\sp{(q-1)/q} }
 \right\} 
\times 
\left\{ \splitfrac{
\left(\displaystyle\int\sb{\mathbb{R}\sp{N}} 
b(x)|u_n|\sp{q} \,\mathrm{d}x \right)\sp{1/q} }
{- \left(\displaystyle\int\sb{\mathbb{R}\sp{N}} 
b(x)|u|\sp{q} \,\mathrm{d}x \right)\sp{1/q} }
\right\}\\
& \qquad 
- \int\sb{\mathbb{R}\sp{N}} 
\left(g(x,u_n)-g(x,u)\right)(u_n - u) \,\mathrm{d}x.
\end{align*}
We remark that all the terms between 
curly brackets in the previous expression
have the same signals; 
therefore, by the limit~\eqref{lim:ggg} we get
\begin{alignat*}{2}
\lim\sb{n\rightarrow\infty} \int\sb{\mathbb{R}\sp{N}} 
|\nabla u_n|\sp{p} \,\mathrm{d}x  
& =\int\sb{\mathbb{R}\sp{N}} |\nabla u|\sp{p}   
\,\mathrm{d}x, 
& \qquad 
\lim\sb{n\rightarrow\infty} 
\int\sb{\mathbb{R}\sp{N}}a(x) | u_n|\sp{p} \,\mathrm{d}x 
& =\int\sb{\mathbb{R}\sp{N}} a(x)|u|\sp{p}   
\,\mathrm{d}x, \\
\shortintertext{and also}
\lim\sb{n\rightarrow\infty} \int\sb{\mathbb{R}\sp{N}} 
|\nabla u_n|\sp{q}   \,\mathrm{d}x 
& = \int\sb{\mathbb{R}\sp{N}} |\nabla u|\sp{q}   
\,\mathrm{d}x,
& \qquad 
\lim\sb{n\rightarrow\infty} 
\int\sb{\mathbb{R}\sp{N}}b(x) | u_n|\sp{q}   
\,\mathrm{d}x 
& = \int\sb{\mathbb{R}\sp{N}} b(x)|u|\sp{q} 
\,\mathrm{d}x.
\end{alignat*}
This implies that 
\begin{alignat*}{2}
\lim\sb{n\rightarrow\infty} 
\left\|u_n\right\|\sp{p}\sb{1,p} 
& = \left\|u\right\|\sp{p}\sb{1,p} 
& \quad \text{and} \quad \lim\sb{n\rightarrow\infty} 
\left\|u_n\right\|\sp{q}\sb{1,q}
& = \left\|u\right\|\sp{q}\sb{1,q}. 
\end{alignat*}
Moreover,
$u_n \rightharpoonup u $ weakly in $E$ 
as \( n \to \infty \);
and finally, 
$u_n \to u $ strongly in $E$ as \( n \to \infty \).
For the details, see  
DiBenedetto~\cite[Proposition V.11.1]{bib:dib}.
\end{proof}

\begin{lem}
\label{lem:convgrad}
Suppose that there exists a sequence 
\( (u\sb{n})\sb{n \in \mathbb{N}} \subset E \)
and a function \( u \in E \) such that 
\( u\sb{n} \to u \) in \( E \)
and \( J'(u\sb{n}) \to 0 \) as \( n \to \infty \). 
Then there exists a subsequence, still denoted in the 
same way, such that
\( \nabla u\sb{n} \to \nabla u \) 
a.\@ e.\@ in \(\mathbb{R}\sp{N}\)
\end{lem}
\begin{proof}
See Assun\c{c}\~{a}o, Carri\~{a}o, 
and Miyagaki~\cite{bib:acmjmaa}
or Benmouloud, Echarghaoui, 
and Sba\"{\i}~\cite{bib:bes}.
\end{proof}

Using 
Lemmas~\ref{lem:I},~\ref{lem:J},~\ref{lem:psbounded},~\ref{lem:funcjps},
and~\ref{lem:convgrad} 
we conclude  that there exists 
$u\in E$  which is a critical point for the 
functional $J$. 
Moreover, this critical point is a positive ground state 
solution to the auxiliary  problem~\eqref{eq:auxprob}, 
that is, \( J(u) = c > 0 \) and \( J'(u) = 0 \).

\section{Estimate for the solution to the auxiliary problem}
\label{sec:estimsol}
In this section we show that the solution to the 
auxiliary problem~\eqref{eq:auxprob} obtained in the 
previous section verifies an important estimate. 
To do this we use several lemmas.

\begin{lem}
\label{lem:estimsol}
For \( R > 1 \), every positive ground state solution 
\( u \) to problem~\eqref{eq:auxprob} 
verifies the estimate
\begin{align*}
\left\|u\right\|\sp{p}\sb{1,p}
+\left\|u\right\|\sp{q}\sb{1,q}
& \leqslant
\dfrac{dkp}{p-1}.
\end{align*}
\end{lem}
\begin{proof}
Combining inequalities~\eqref{c:d},~\eqref{J:inf} 
and~\eqref{J:infk}, it follows that
\begin{equation*}
\dfrac{(p-1)}{kp} \left\{
 \left\|u\right\|\sp{p}\sb{1,p}
+\left\|u\right\|\sp{q}\sb{1,q}\right\}
\leqslant
J(u) - \dfrac{1}{\theta}\,J'(u)u  
 = J(u) = c \leqslant d.
\end{equation*}
The conclusion of the lemma follows immediately.
\end{proof}

We remark that the boundedness of the norm of the 
ground state solution 
to problem~\eqref{eq:auxprob} shown in 
Lemma~\ref{lem:estimsol} depends only 
on the potential 
functions \(a\sb{\infty}\) and \(b\sb{\infty}\), 
on the nonlinearity \(f\) and 
on the constant \(\theta\); 
it is independ of the constant \( R > 1 \).

The next lemma is a crucial step to establish an 
important 
estimate involving the
norm of the solution to the auxiliary 
problem~\eqref{eq:auxprob}
in the space \( L\sp{\infty}(\mathbb{R}\sp{N}) \). 
To prove it we adapt the arguments by 
Alves and Souto~\cite{bib:as};
see also 
Gilbarg and Trudinger~\cite[Section~8.6]{bib:gt},
Br\'{e}zis and Kato~\cite{bib:bk},
Pucci and Servadei~\cite{bib:ps}, 
and
Bastos, Miyagaki, and Vieira~\cite{bib:bmv}.

\begin{lem}
\label{lem:asbk}
Suppose that \(p, r \in \mathbb{R}\) verify 
the inequality \( pr > N \). 
Let
\( H \colon \mathbb{R}\sp{N} \times \mathbb{R} 
\to \mathbb{R} \) 
be a continuous function such that
\( |H(x,s)| \leqslant h(x)|s|\sp{p-2}s \)
for all \( s > 0 \)
with the function
\( h \colon \mathbb{R}\sp{N} \to \mathbb{R}\) 
so that
\( h \in L\sp{r}(\mathbb{R}\sp{N}) \) 
and let
\( A, B \colon \mathbb{R}\sp{N} \to \mathbb{R}\) 
be nonnegative functions.
Suppose also that
\( v \in E \subset D\sp{1,p}(\mathbb{R}\sp{N}) \cap
D\sp{1,q}(\mathbb{R}\sp{N})\) 
is a weak solution to the problem
\begin{align}
\label{prob:h}
-\Delta\sb{p} v 
-\Delta\sb{q} v 
+ A(x) |v|\sp{p-2}v
+ B(x) |v|\sp{q-2}v = H(x,v), \qquad 
x \in \mathbb{R}\sp{N}.
\end{align}
Then there exists a constant
$ M\sb{1} = M\sb{1}( N, p, q, r, 
\| h \|\sb{L\sp{r}(\mathbb{R}\sp{N})}) > 0 $, 
which does not depend on the functions 
\( A \) and \( B \),
such that
\begin{align*}
\| v \|\sb{L\sp{\infty}(\mathbb{R}\sp{N})}
& \leqslant M\sb{1} 
\max\left\{
\| v \|\sb{L\sp{p\sp{*}}(\mathbb{R}\sp{N})},
K, KL\sb{v}, 1 \right\},
\end{align*}
where \( K \) and \( L\sb{v} \) 
are defined by~\eqref{eq:defk} 
and by~\eqref{eq:defl}, respectively.
\end{lem}
\begin{proof}
Let \( \beta > 1 \); 
for every \( m \in \mathbb{N} \) we define the subsets
\begin{align*}
A\sb{m} & \equiv \{ x \in \mathbb{R}\sp{N} \colon 
1 < | v(x) |\sp{\beta -1} \leqslant m \}; \\
B\sb{m} & \equiv \{ x \in \mathbb{R}\sp{N} \colon 
| v(x) |\sp{\beta -1} > m \}; \\
C\sb{m} & \equiv \{ x \in \mathbb{R}\sp{N} \colon 
| v(x) |\sp{\beta -1} \leqslant 1 \}.
\end{align*}
We also define the sequence of functions 
\( ( v\sb{m} )\sb{m \in \mathbb{N}}  
\subset 
D\sp{1,p}(\mathbb{R}\sp{N}) \cap  
D\sp{1,q}(\mathbb{R}\sp{N}) \) 
by
\begin{align*}
v\sb{m}(x) \equiv
\begin{cases}
|v(x)|\sp{p(\beta -1)}v(x), 
& \text{if $x \in A\sb{m}$}; \\
m\sp{p} v(x), 
& \text{if $x \in B\sb{m}$}; \\
|v(x)|\sp{q(\beta -1)}v(x),
& \text{if $x \in C\sb{m}$}.
\end{cases}
\end{align*}
It is easy to verify that for every 
\( x \in \mathbb{R}\sp{N} \) we have
\( v\sb{m}(x) 
\leqslant \max \left\{ 
\left| v(x) \right|\sp{p(\beta - 1) +1}, 
\left| v(x) \right|\sp{q(\beta - 1) +1}
\right\} \).
Additionally, simple computations show that
\begin{align*}
\nabla v\sb{m}(x) =
\begin{cases}
\left( p(\beta -1)+1 \right) 
\left| v(x) \right|\sp{p(\beta -1)} 
\nabla v(x), 
& \text{if $x \in A\sb{m}$;}\\
m\sp{p} \nabla v(x), 
& \text{if $x \in B\sb{m}$;}\\
\left( q(\beta -1)+1 \right) 
\left| v(x) \right|\sp{q(\beta -1)} 
\nabla v(x), 
& \text{if $x \in C\sb{m}$.}
\end{cases}
\end{align*}
Furthermore, 
\( ( v\sb{m} )\sb{m \in \mathbb{N}} \subset E \). 
Indeed, 
\begin{align*}
\int\sb{\mathbb{R}\sp{N}} 
a(x) |v\sb{m}|\sp{p}\,\mathrm{d}x  
& \leqslant
\int\sb{A\sb{m}} 
a(x) \left( |v|\sp{p-1} v \right)
m\sp{p(p-1)+p} \, \mathrm{d}x 
+ \int\sb{B\sb{m}} 
a(x) |v|\sp{p-1} v m\sp{p(p-1)+p}  \, \mathrm{d}x \\
& \qquad + \int\sb{C\sb{m}} 
a(x) \left( |v|\sp{p-1} v \right) \, \mathrm{d}x \\
& \leqslant 
m\sp{p\sp{2}} 
\int\sb{\mathbb{R}\sp{N}} 
a(x) |v|\sp{p-1} v \,\mathrm{d}x 
< + \infty.
\end{align*}
And in a similar way, we have
\begin{align*}
\int\sb{\mathbb{R}\sp{N}} 
b(x) |v\sb{m}|\sp{q},\mathrm{d}x 
& = m\sp{pq} 
\int\sb{\mathbb{R}\sp{N}} 
b(x) |v|\sp{q-1} v \,\mathrm{d}x 
< + \infty.
\end{align*}

Multiplying both sides of the 
differential equation~\eqref{prob:h}
by the test function \( v\sb{m} \) and integrating
the left-hand with the help of the divergence theorem, 
we deduce that
\begin{align*}
{} & \int\sb{\mathbb{R}\sp{N}} 
| \nabla v |\sp{p-2} \nabla v \cdot 
\nabla v\sb{m}\,\mathrm{d}x
+ \int\sb{\mathbb{R}\sp{N}} 
| \nabla v |\sp{q-2} \nabla v \cdot 
\nabla v\sb{m}\,\mathrm{d}x \\
& \qquad + \int\sb{\mathbb{R}\sp{N}} 
A(x) |v|\sp{p-2} v v\sb{m}\,\mathrm{d}x
+ \int\sb{\mathbb{R}\sp{N}} 
B(x) |v|\sp{q-2} v v\sb{m}\,\mathrm{d}x \\
& \qquad \qquad = 
\int\sb{\mathbb{R}\sp{N}} 
H(x,v) v\sb{m}\,\mathrm{d}x
\end{align*}
Using the definition of the function \( v\sb{m} \), 
we obtain
\begin{align}
{} &
\left( p(\beta -1)+1 \right) \left\{ 
\int\sb{A\sb{m}} 
| \nabla v |\sp{p} |v|\sp{p(\beta -1)} \,\mathrm{d}x 
+ \int\sb{A\sb{m}} 
| \nabla v |\sp{q} |v|\sp{p(\beta -1)} 
\,\mathrm{d}x \right\} \nonumber\\
& \qquad
+ \left( q(\beta -1)+1 \right) \left\{ 
\int\sb{C\sb{m}} 
| \nabla v |\sp{p} |v|\sp{q(\beta -1)} \,\mathrm{d}x
+ \int\sb{C\sb{m}} 
| \nabla v |\sp{q} |v|\sp{q(\beta -1)} 
\,\mathrm{d}x \right\} \nonumber\\
& \qquad \qquad =
\int\sb{\mathbb{R}\sp{N}} 
| \nabla v |\sp{p-2} \nabla v \cdot \nabla v\sb{m}
\,\mathrm{d}x 
+\int\sb{\mathbb{R}\sp{N}} 
| \nabla v |\sp{q-2} \nabla v \cdot \nabla v\sb{m}
\,\mathrm{d}x \nonumber \\
& \qquad \qquad \qquad 
- m\sp{p} \left\{
\int\sb{B\sb{m}} 
| \nabla v |\sp{p} \,\mathrm{d}x 
+ \int\sb{B\sb{m}} 
| \nabla v |\sp{q} \,\mathrm{d}x \right\} \nonumber \\
& \qquad \qquad \leqslant 
\int\sb{\mathbb{R}\sp{N}} 
| \nabla v |\sp{p-2} \nabla v \cdot \nabla v\sb{m}
\,\mathrm{d}x 
+ \int\sb{\mathbb{R}\sp{N}} 
A(x) |v|\sp{p-2}v v\sb{m}\,\mathrm{d}x \nonumber \\
& \qquad \qquad \qquad 
+\int\sb{\mathbb{R}\sp{N}} 
| \nabla v |\sp{q-2} \nabla v \cdot \nabla v\sb{m}
\,\mathrm{d}x
+\int\sb{\mathbb{R}\sp{N}} 
B(x) |v|\sp{q-2} v v\sb{m}\,\mathrm{d}x.
\label{eq:group}
\end{align}

Now we define another sequence of functions 
\( (w\sb{m} )\sb{m \in \mathbb{N}} \subset E \) by 
\begin{align*}
w\sb{m}(x) = 
\begin{cases}
| v(x) |\sp{\beta -1}v(x), 
& \text{if $x \in A\sb{m} \cup C\sb{m}$}; \\
m v(x), 
& \text{if $x \in B\sb{m}$}.
\end{cases}
\end{align*}
Direct computations show that
\begin{align*}
\nabla w\sb{m}(x) =
\begin{cases}
\beta | v(x) |\sp{\beta -1} \nabla v(x), 
& \text{if $x \in A\sb{m} \cup C\sb{m}$}; \\
m \nabla v(x),
& \text{if $x \in B\sb{m}$}.
\end{cases}
\end{align*}
Using the hypothesis 
\( 2 \leqslant q \leqslant p < N \), we obtain
\begin{align*}
{} & 
\int\sb{\mathbb{R}\sp{N}} 
|\nabla w\sb{m}|\sp{p}\,\mathrm{d}x 
+\int\sb{\mathbb{R}\sp{N}} 
A(x) |w\sb{m}|\sp{p}\,\mathrm{d}x 
- \int\sb{\mathbb{R}\sp{N}} | \nabla v|\sp{p-2} 
\nabla v \cdot \nabla v\sb{m}\,\mathrm{d}x 
- \int\sb{\mathbb{R}\sp{N}} A(x) |v|\sp{p-2}v v\sb{m} 
\, \mathrm{d}x \\
& \qquad +\int\sb{\mathbb{R}\sp{N}} 
|\nabla w\sb{m}|\sp{q}\,\mathrm{d}x 
+\int\sb{\mathbb{R}\sp{N}} B(x) |w\sb{m}|\sp{q}
\,\mathrm{d}x 
- \int\sb{\mathbb{R}\sp{N}} | \nabla v |\sp{q-2} 
\nabla v \cdot \nabla v\sb{m}\,\mathrm{d}x 
- \int\sb{\mathbb{R}\sp{N}} B(x) |v|\sp{q-2}v v\sb{m} 
\, \mathrm{d}x \\
& \qquad \qquad \leqslant 
\beta\sp{p} \int\sb{A\sb{m} \cup C\sb{m}} 
|\nabla v|\sp{p} |v|\sp{p(\beta -1)} \,\mathrm{d}x 
+ \beta\sp{p} \int\sb{A\sb{m} \cup C\sb{m}} 
|\nabla v|\sp{q} |v|\sp{q(\beta -1)} 
\,\mathrm{d}x \\
& \qquad \qquad \qquad 
-\left( p(\beta -1)+1 \right) \left\{
\int\sb{A\sb{m}} 
|\nabla v|\sp{p} |v|\sp{p(\beta -1)} \,\mathrm{d}x  
+ \int\sb{A\sb{m}} 
|\nabla v|\sp{q} |v|\sp{p(\beta -1)} \,\mathrm{d}x 
\right\} \\
& \qquad \qquad \qquad 
+ \int\sb{A\sb{m}} B(x) \left( 
|v|\sp{q\beta} - |v|\sp{p(\beta -1)+q}\right)
\, \mathrm{d}x
+ \int\sb{C\sb{m}} A(x) 
\left( |v|\sp{p\beta} - |v|\sp{p+q(\beta -1)}\right)
\, \mathrm{d}x \\
& \qquad \qquad \qquad
+ \left( m\sp{q} - m\sp{p} \right) 
\int\sb{B\sb{m}} B(x) |v|\sp{q}\,\mathrm{d}x \\
& \qquad \qquad \qquad
-\left( q(\beta -1)+1 \right) \left\{ 
\int\sb{C\sb{m}} 
| \nabla v|\sp{p} |v|\sp{q(\beta -1)} \,\mathrm{d}x 
+ \int\sb{C\sb{m}} 
| \nabla v|\sp{q} |v|\sp{q(\beta -1)} \,\mathrm{d}x 
\right\} \\
& \qquad \qquad \qquad 
+ \left( m\sp{q} - m\sp{p} \right) 
\int\sb{B\sb{m}} |\nabla v|\sp{q}\,\mathrm{d}x.
\end{align*}
And after we get rid of the non positive terms, 
we can regroup the expressions to obtain
\begin{align*}
{} & 
\int\sb{R\sp{N}} |\nabla w\sb{m}|\sp{p}\,\mathrm{d}x 
+\int\sb{\mathbb{R}\sp{N}} A(x) |w\sb{m}|\sp{p}
\,\mathrm{d}x 
+ \int\sb{R\sp{N}} |\nabla w\sb{m}|\sp{q}\,\mathrm{d}x 
+ \int\sb{\mathbb{R}\sp{N}} B(x) |w\sb{m}|\sp{q}
\,\mathrm{d}x \\
&  \qquad
= \left( \beta\sp{p} 
- \left( p(\beta -1)+1 \right) \right) 
\int\sb{A\sb{m}} |\nabla v|\sp{p}
|v|\sp{p(\beta -1)} \,\mathrm{d}x 
+ \beta\sp{p} \int\sb{C\sb{m}} |\nabla v|\sp{p}
|v|\sp{p(\beta -1)}\,\mathrm{d}x 
\\
& \qquad \qquad 
+ \left( \beta\sp{q} 
- \left( q(\beta -1)+1 \right) \right)
\int\sb{C\sb{m}} |\nabla v|\sp{q}
|v|\sp{q(\beta -1)} \,\mathrm{d}x 
+ \beta\sp{q} \int\sb{A\sb{m}} |\nabla v|\sp{q}
|v|\sp{q(\beta -1)} \,\mathrm{d}x \\
& \qquad \qquad
+ \int\sb{\mathbb{R}\sp{N}} | \nabla v|\sp{p-2} 
\nabla v \cdot \nabla v\sb{m}\,\mathrm{d}x 
+ \int\sb{\mathbb{R}\sp{N}} A(x) |v|\sp{p-2}v v\sb{m} 
\, \mathrm{d}x \\
& \qquad \qquad +\int\sb{\mathbb{R}\sp{N}} 
| \nabla v |\sp{q-2} \nabla v \cdot \nabla v\sb{m}
\,\mathrm{d}x 
+\int\sb{\mathbb{R}\sp{N}} B(x) |v|\sp{q-2}v v\sb{m} 
\, \mathrm{d}x 
\end{align*}

So, using inequality~\eqref{eq:group} we deduce that
\begin{align*}
{} & 
\int\sb{R\sp{N}} |\nabla w\sb{m}|\sp{p}\,\mathrm{d}x 
+\int\sb{\mathbb{R}\sp{N}} A(x) |w\sb{m}|\sp{p}
\,\mathrm{d}x 
+ \int\sb{R\sp{N}} |\nabla w\sb{m}|\sp{q}\,\mathrm{d}x 
+ \int\sb{\mathbb{R}\sp{N}} B(x) |w\sb{m}|\sp{q}
\,\mathrm{d}x \\
& \qquad
\leqslant \left(\dfrac{\beta\sp{p}}
{q(\beta -1)+1} 
\right) 
\left\{ 
\splitfrac{\displaystyle\int\sb{\mathbb{R}\sp{N}} 
| \nabla v|\sp{p-2} \nabla v \cdot \nabla v\sb{m}
\, \mathrm{d}x 
+ \displaystyle\int\sb{\mathbb{R}\sp{N}} 
A(x) |v|\sp{p-2}v v\sb{m} \, \mathrm{d}x }
{ + \displaystyle\int\sb{\mathbb{R}\sp{N}} 
| \nabla v|\sp{q-2} \nabla v \cdot \nabla v\sb{m}
\,\mathrm{d}x 
+ \displaystyle\int\sb{\mathbb{R}\sp{N}} 
B(x) |v|\sp{q-2}v v\sb{m} \, \mathrm{d}x }
\right\} \\
& \qquad \qquad
+ \beta\sp{p} \int\sb{C\sb{m}} 
|\nabla v|\sp{p} |v|\sp{p(\beta -1)} \,\mathrm{d}x 
+ \beta\sp{q} \int\sb{A\sb{m}} 
|\nabla v|\sp{q} |v|\sp{q(\beta -1)}\,\mathrm{d}x.
\end{align*}

Now we estimate some integrals that appear in the 
previous inequality. 
First, by definition of \( A\sb{m} \) we have
\begin{align*}
\int\sb{A\sb{m}} |\nabla v|\sp{q} |v|\sp{q(\beta -1)}
\,\mathrm{d}x
& =
\int\sb{A\sb{m}} 
\dfrac{|\nabla v|\sp{q-2}}
{[p(\beta -1)+1]|v|\sp{(p-q)(\beta -1)}} \,
\nabla v \cdot \nabla v\sb{m} \,\mathrm{d}x \\
& \leqslant
\int\sb{\mathbb{R}\sp{N}} | \nabla v|\sp{p-2} 
\nabla v \cdot \nabla v\sb{m}\,\mathrm{d}x 
+ \int\sb{\mathbb{R}\sp{N}} A(x) |v|\sp{p-2}v v\sb{m} 
\, \mathrm{d}x \\
& \qquad +
\int\sb{\mathbb{R}\sp{N}} | \nabla v|\sp{q-2} 
\nabla v \cdot \nabla v\sb{m}\,\mathrm{d}x 
+ \int\sb{\mathbb{R}\sp{N}} B(x) |v|\sp{q-2}v v\sb{m} 
\, \mathrm{d}x.
\end{align*}
In a similar way, by definition of 
\( C\sb{m} \) we have
\begin{align*}
\int\sb{C\sb{m}} |\nabla v|\sp{p} |v|\sp{p(\beta -1)}
\,\mathrm{d}x
& \leqslant
\int\sb{\mathbb{R}\sp{N}} | \nabla v|\sp{p-2} 
\nabla v \cdot \nabla v\sb{m}\,\mathrm{d}x 
+ \int\sb{\mathbb{R}\sp{N}} A(x) |v|\sp{p-2}v v\sb{m} 
\, \mathrm{d}x \\
& \qquad +
\int\sb{\mathbb{R}\sp{N}} | \nabla v|\sp{q-2} 
\nabla v \cdot \nabla v\sb{m}\,\mathrm{d}x 
+ \int\sb{\mathbb{R}\sp{N}} B(x) |v|\sp{q-2}v v\sb{m} 
\, \mathrm{d}x.
\end{align*}
Using these inequalities we deduce that
\begin{align*}
{} & 
\int\sb{R\sp{N}} |\nabla w\sb{m}|\sp{p}\,\mathrm{d}x 
+ \int\sb{\mathbb{R}\sp{N}} A(x) |w\sb{m}|\sp{p}
\,\mathrm{d}x 
+ \int\sb{R\sp{N}} |\nabla w\sb{m}|\sp{q}\,\mathrm{d}x 
+\int\sb{\mathbb{R}\sp{N}} B(x) |w\sb{m}|\sp{q}
\,\mathrm{d}x \\
&  \qquad
\leqslant \left( 
\beta\sp{p} + \dfrac{\beta\sp{p}}
{q(\beta -1)+1} 
\right) 
\left\{ 
\splitfrac{\displaystyle\int\sb{\mathbb{R}\sp{N}} 
| \nabla v|\sp{p-2} \nabla v \cdot \nabla v\sb{m}
\,\mathrm{d}x 
+\displaystyle\int\sb{\mathbb{R}\sp{N}} 
A(x) |v|\sp{p-2}v v\sb{m} \, \mathrm{d}x }
{+\displaystyle\int\sb{\mathbb{R}\sp{N}} 
| \nabla v|\sp{q-2} \nabla v \cdot \nabla v\sb{m}
\,\mathrm{d}x 
+\displaystyle\int\sb{\mathbb{R}\sp{N}} 
B(x) |v|\sp{q-2}v v\sb{m} \, \mathrm{d}x }
\right\} \\
&  \qquad
\leqslant 2\beta\sp{p} 
\left\{ \splitfrac{
\displaystyle\int\sb{\mathbb{R}\sp{N}} 
| \nabla v|\sp{p-2} \nabla v \cdot \nabla v\sb{m}
\,\mathrm{d}x 
+\int\sb{\mathbb{R}\sp{N}} A(x) |v|\sp{p-2}v v\sb{m} 
\, \mathrm{d}x }
{+\displaystyle\int\sb{\mathbb{R}\sp{N}} 
| \nabla v|\sp{q-2} \nabla v \cdot \nabla v\sb{m}
\,\mathrm{d}x 
+\displaystyle\int\sb{\mathbb{R}\sp{N}} 
B(x) |v|\sp{q-2}v v\sb{m} \, \mathrm{d}x }
\right\} \\
& \qquad 
= 2\beta\sp{p} 
\int\sb{\mathbb{R}\sp{N}} H(x,v)v\sb{m}\,\mathrm{d}x.
\end{align*}

Using the Sobolev inequality~\eqref{eq:Sob} 
and the hypothesis 
\( H(x,s) \leqslant h(x)|s|\sp{p-1} \), we obtain
\begin{align*}
{} & 
\left(
\int\sb{A\sb{m} \cup C\sb{m}} 
|w\sb{m}|\sp{p\sp{*}} \,\mathrm{d}x 
\right)\sp{p/p\sp{*}}  
\leqslant 
\left(
\int\sb{\mathbb{R}\sp{N}} 
|w\sb{m}|\sp{p\sp{*}} \,\mathrm{d}x 
\right)\sp{p/p\sp{*}} \\
& \qquad \leqslant 
S \int\sb{\mathbb{R}\sp{N}} 
|\nabla w\sb{m}|\sp{p} \,\mathrm{d}x 
\leqslant 
S \left\{ 
\splitfrac{\displaystyle\int\sb{\mathbb{R}\sp{N}} 
|\nabla w\sb{m}|\sp{p} \,\mathrm{d}x 
+ \displaystyle\int\sb{\mathbb{R}\sp{N}} 
a(x)|w\sb{m}|\sp{p} \,\mathrm{d}x }
{ + \displaystyle\int\sb{\mathbb{R}\sp{N}}
|\nabla w\sb{m}|\sp{q} \,\mathrm{d}x 
+ \displaystyle\int\sb{\mathbb{R}\sp{N}} 
b(x)|w\sb{m}|\sp{q} \,\mathrm{d}x }
\right\} \\
& \qquad \leqslant 
2S \beta\sp{p} \int\sb{\mathbb{R}\sp{N}} 
H(x,v) v\sb{m}\,\mathrm{d}x 
\leqslant 
2S \beta\sp{p} \int\sb{\mathbb{R}\sp{N}} h(x)
\left| v \right|\sp{p-1} v\sb{m}\,\mathrm{d}x \\
& \qquad =  
2S \beta\sp{p} 
\left\{ \splitfrac{\displaystyle\int\sb{A\sb{m}} h(x)
\left| v \right|\sp{p-2}v 
\left| v \right|\sp{p(\beta -1)} v \,\mathrm{d}x 
+ \displaystyle\int\sb{B\sb{m}} h(x)
\left| v \right|\sp{p-2}v 
\, m\sp{p} v \,\mathrm{d}x } 
{ + \displaystyle\int\sb{C\sb{m}} h(x)
\left| v \right|\sp{p-2}v 
\left| v \right|\sp{q(\beta -1)} v \,\mathrm{d}x }
\right\} \\
& \qquad  \leqslant 
2S \beta\sp{p} 
\left\{ \displaystyle\int\sb{\mathbb{R}\sp{N}} h(x)
\left| v \right|\sp{p\beta} \,\mathrm{d}x 
 + \displaystyle\int\sb{\mathbb{R}\sp{N}} h(x)
\left| v \right|\sp{p} \,\mathrm{d}x
\right\}, 
\end{align*}
where in the last passage we used the definitions of the functions \(v\sb{m}\) and \(w\sb{m}\), together with the facts that in \(B\sb{m}\) we have \( |w\sb{m}|\sp{p} \leqslant |v|\sp{p\beta}\) and in \(C\sb{m}\) we have \( |v|\sp{p+q(\beta -1)} \leqslant |v|\sp{p}\).

Passing to the limit as \( m \to \infty \) 
and using Lebesgue's dominated convergence theorem, 
it follows that
\begin{align*}
\left( \int\sb{\mathbb{R}\sp{N}} 
|v|\sp{p\sp{*}\beta} \,\mathrm{d}x \right)\sp{p/p\sp{*}}
& \leqslant
2S \beta\sp{p} 
\left\{ \displaystyle\int\sb{\mathbb{R}\sp{N}} h(x)
\left| v \right|\sp{p\beta} \,\mathrm{d}x 
 + \displaystyle\int\sb{\mathbb{R}\sb{N}} h(x)
\left| v \right|\sp{p} \,\mathrm{d}x
\right\}.
\end{align*}

Applying H\"{o}lder's inequality to both terms 
on the right-hand side of the previous inequality, 
we obtain
\begin{align*}
\int\sb{\mathbb{R}\sp{N}} h(x)
\left| v \right|\sp{p\beta} \,\mathrm{d}x  
& \leqslant 
\| h \|\sb{L\sp{r}(\mathbb{R}\sp{N})} 
\| v \|\sb{L\sp{p \beta r'}(\mathbb{R}\sp{N})}
\sp{p \beta} \\
\intertext{and}
\int\sb{\mathbb{R}\sb{N}} h(x)
\left| v \right|\sp{p} \,\mathrm{d}x
& \leqslant
\| h \|\sb{L\sp{r}(\mathbb{R}\sp{N})} 
\| v \|\sb{L\sp{p r'}(\mathbb{R}\sp{N})}\sp{p};
\end{align*}
hence
\begin{align*}
\| v \|\sb{L\sp{p\sp{*} \beta}(\mathbb{R}\sp{N})}
\sp{p \beta}
& \leqslant
2S \| h \|\sb{L\sp{r}(\mathbb{R}\sp{N})} 
\beta\sp{p} \left\{ 
\| v \|\sb{L\sp{p \beta r'}(\mathbb{R}\sp{N})}
\sp{p \beta}
+ 
\| v \|\sb{L\sp{p r'}(\mathbb{R}\sp{N})}\sp{p}
\right\} \\
& \leqslant 
2S \| h \|\sb{L\sp{r}(\mathbb{R}\sp{N})} 
\beta\sp{p} \left\{ 
\max\left\{
\| v \|\sb{L\sp{p \beta r'}
(\mathbb{R}\sp{N})}\sp{p \beta} , 1 \right\}
+ 
\max\left\{\| v \|\sb{L\sp{p r'}
(\mathbb{R}\sp{N})}\sp{p} , 1 \right\}
\right\} \\
& = 
C\sb{1}\sp{p} \beta\sp{p} \max\left\{ 
\| v \|\sb{L\sp{p \beta r'}(\mathbb{R}\sp{N})}
\sp{p \beta},
\max\left\{\| v \|\sb{L\sp{p r'}
(\mathbb{R}\sp{N})}\sp{p} , 1 \right\}
\right\}, 
\end{align*}
where we used the notation
\( C\sb{1}\sp{p} 
= C\sb{1}\sp{p}(N,p,q,r,
\| h \|\sb{L\sp{r}(\mathbb{R}\sp{N})}) 
\equiv 4S \| h \|\sb{L\sp{r}(\mathbb{R}\sp{N})} 
> 0\).

Writing
\( \beta = \sigma\sp{j} \) for \( j \in \mathbb{N} \)
we deduce that 
\begin{align}
\label{eq:iteration}
\| v \|\sb{L\sp{p\sp{*} \sigma\sp{j}}(\mathbb{R}\sp{N})}
& \leqslant
C\sb{1}\sp{1/\sigma\sp{j}} \sigma\sp{j/\sigma\sp{j}} 
\max\left\{ 
\| v \|\sb{L\sp{p \sigma\sp{j} r'}(\mathbb{R}\sp{N})},
\max\left\{\| v \|\sb{L\sp{p r'}
(\mathbb{R}\sp{N})}\sp{1/\sigma\sp{j}} , 1 \right\}
\right\}.
\end{align}
Choosing \( \sigma = p\sp{*}/pr' > 1\),
from inequality~\eqref{eq:iteration} with 
\( j = 1 \) we obtain
\begin{align*}
\| v \|\sb{L\sp{p\sp{*} \sigma}(\mathbb{R}\sp{N})}
& \leqslant
C\sb{1}\sp{1/\sigma} \sigma\sp{1/\sigma} 
\max\left\{ 
\| v \|\sb{L\sp{p\sp{*}}(\mathbb{R}\sp{N})},
\max\left\{\| v \|\sb{L\sp{p r'}
(\mathbb{R}\sp{N})}\sp{1/\sigma} , 1 \right\}
\right\};
\end{align*}
and from inequality~\eqref{eq:iteration} 
with \( j = 2 \)
together with the previous inequality we obtain
\begin{align*}
\| v \|\sb{L\sp{p\sp{*} \sigma\sp{2}}(\mathbb{R}\sp{N})}
& \leqslant
C\sb{1}\sp{1/\sigma\sp{2}} \sigma\sp{2/\sigma\sp{2}} 
\max\left\{ 
\| v \|\sb{L\sp{p\sp{*} \sigma}(\mathbb{R}\sp{N})},
\max\left\{\| v \|\sb{L\sp{p r'}
(\mathbb{R}\sp{N})}\sp{1/\sigma\sp{2}} , 1 \right\}
\right\} \nonumber \\
& \leqslant
C\sb{1}\sp{1/\sigma\sp{2}} \sigma\sp{2/\sigma\sp{2}} 
\max\left\{ \splitfrac{ 
C\sb{1}\sp{1/\sigma} \sigma\sp{1/\sigma} 
\max\left\{ 
\| v \|\sb{L\sp{p\sp{*}}(\mathbb{R}\sp{N})},
\max\left\{\| v \|\sb{L\sp{p r'}
(\mathbb{R}\sp{N})}\sp{1/\sigma} , 1 \right\}
\right\},}
{\max\left\{\| v \|\sb{L\sp{p r'}
(\mathbb{R}\sp{N})}\sp{1/\sigma\sp{2}} , 1 \right\}}
\right\} \nonumber \\
& \leqslant
C\sb{1}\sp{1/\sigma + 1/\sigma\sp{2}} 
\sigma\sp{1/\sigma + 2/\sigma\sp{2}} \nonumber \\
& \qquad \times
\max\left\{ \splitfrac{ 
\| v \|\sb{L\sp{p\sp{*}}(\mathbb{R}\sp{N})},
\max\left\{
\big( C\sb{1}\sp{1/\sigma} 
\sigma\sp{1/\sigma} \big)\sp{-1}, 1 \right\},}
{
\max\left\{
\big( C\sb{1}\sp{1/\sigma} 
\sigma\sp{1/\sigma} \big)\sp{-1}, 1 \right\}
\max\left\{\| v \|\sb{L\sp{p r'}
(\mathbb{R}\sp{N})}\sp{1/\sigma} ,
\| v \|\sb{L\sp{p r'}
(\mathbb{R}\sp{N})}\sp{1/\sigma\sp{2}} , 1 \right\}}
\right\}.
\end{align*}
Proceeding in this way, for \( j \in \mathbb{N} \)
we obtain
\begin{align}
\label{eq:iterationj}
\| v \|\sb{L\sp{p\sp{*} \sigma\sp{j}}(\mathbb{R}\sp{N})} 
& \leqslant
C\sb{1}\sp{s\sb{j}} 
\sigma\sp{t\sb{j}} 
\max\left\{ 
\| v \|\sb{L\sp{p\sp{*}}(\mathbb{R}\sp{N})},
K\sb{j}, K\sb{j}L\sb{j} \right\},
\end{align}
where 
\( s\sb{j} \equiv 1/\sigma + 1/\sigma\sp{2} + \cdots +
1/\sigma\sp{j} \); 
\( t\sb{j} \equiv 1/\sigma + 2/\sigma\sp{2} + \cdots +
j/\sigma\sp{j} \); 
\begin{align*} 
K\sb{j} 
& \equiv 
\begin{cases}
1, & \text{ if $j = 1$}; \\
\max\sb{1 \leqslant i \leqslant j-1}
\left\{ 
C\sb{1}\sp{-s\sb{i}} \sigma\sp{-t\sb{i}}, 1
\right\},
& \text{ if $j \geqslant 2$};
\end{cases} 
\end{align*}
and 
\begin{align*}
L\sb{j} 
& \equiv 
\max\sb{1 \leqslant i \leqslant j}
\left\{  
\| v \|\sb{L\sp{p r'}
(\mathbb{R}\sp{N})}\sp{1/\sigma\sp{i}} , 1 
\right\}.
\end{align*}
Since \( \sigma > 1 \), we have
\( \lim\sb{j \to \infty} s\sb{j} 
= 1/(\sigma -1) \)
and
\( \lim\sb{j \to \infty} t\sb{j} 
= \sigma/(\sigma -1)\sp{2} \);
hence,
\begin{align}
\label{eq:defk}
\lim\sb{j \to \infty} K\sb{j} 
& \equiv K =
\begin{cases}
\big( C\sb{1}\sp{1/(\sigma -1)} 
\sigma\sp{\sigma/(\sigma -1) \sp{2}} \big)\sp{-1} 
& \text{ if $C\sb{1} \leqslant 1$}; \\
\big( C\sb{1}\sp{1/\sigma} 
\sigma\sp{1/\sigma} \big)\sp{-1} 
& \text{ if $C\sb{1} > 1$}; 
\end{cases}
\end{align}
and
\begin{align}
\label{eq:defl}
\lim\sb{j \to \infty} L\sb{j} 
& \equiv L\sb{v} = 
\begin{cases} 
1, & \text{ if $ \| v \|\sb{L\sp{p r'}
(\mathbb{R}\sp{N})} \leqslant 1 $}; \\
\| v \|\sb{L\sp{p r'}(\mathbb{R}\sp{N})}
\sp{1/(\sigma -1)},
& \text{ if $ \| v \|\sb{L\sp{p r'}(\mathbb{R}\sp{N})} 
> 1 $ }.
\end{cases}
\end{align}
Using the fact that 
\( v \in E \subset D\sp{1,p}(\mathbb{R}\sp{N}) 
\cap D\sp{1,q}(\mathbb{R}\sp{N}) \), 
applying H\"{o}lder's inequality we deduce that 
\( L\sb{v} < +\infty \).

Finally, passing to the limit as \( j \to \infty\) 
and using inequality~\eqref{eq:iterationj} we obtain
\begin{align}
\label{eq:estimatelinf}
\left\| v \right\|\sb{L\sp{\infty}(\mathbb{R}\sp{N})}
& = 
\lim\sb{j \to \infty} 
\left\| v \right\|\sb{L\sp{p\sp{*} 
\sigma\sp{j}}(\mathbb{R}\sp{N})} \nonumber \\
& \leqslant 
C\sb{1}\sp{1/(\sigma - 1)} 
\sigma\sp{\sigma/(\sigma -1 )\sp{2}}
\max\left\{ 
\left\| v \right\|\sb{L\sp{p\sp{*}}(\mathbb{R}\sp{N})}, 
K, KL\sb{v}, 1 \right\} \nonumber \\
& \equiv M\sb{1} 
\max\left\{ 
\left\| v \right\|\sb{L\sp{p\sp{*}}(\mathbb{R}\sp{N})}, 
K, KL\sb{v}, 1 \right\},
\end{align}
where 
\( M\sb{1} = M\sb{1}(N,p,q,r, 
\| h \|\sb{L\sp{r}(\mathbb{R}\sp{N})}) \).
This concludes the proof of the lemma.
\end{proof}

\begin{lem}
\label{lem:estiminf}
For every \( R > 1 \) there exist a constant
\( M\sb{2} = M\sb{2}(N, p, q, r, 
a\sb{\infty}, b\sb{\infty}, \theta, c\sb{0}) \) 
such that
any positive ground state solution 
\( u \in 
D\sp{1,p}(\mathbb{R}\sp{N})\cap 
D\sp{1,q}(\mathbb{R}\sp{N})\) to the 
auxiliary problem~\eqref{eq:auxprob} 
verifies the inequality
\begin{align*}
\| u \|\sb{L\sp{\infty}(\mathbb{R}\sp{N})}
& \leqslant
M\sb{2}.
\end{align*}
\end{lem}
\begin{proof}
Consider \( R > 1 \) and let 
\( u \in 
D\sp{1,p}(\mathbb{R}\sp{N})\cap 
D\sp{1,q}(\mathbb{R}\sp{N}) \) 
be a positive ground state solution to the auxiliary 
problem~\eqref{eq:auxprob}. 
Now we define the function 
\( H \colon \mathbb{R}\sp{N} \times \mathbb{R} 
\to \mathbb{R}\) by
\begin{equation*}
H(x,t) \equiv
\begin{cases}
f(t), & \text{if $|x| \leqslant R$
or if $|x| > R$ and 
$f(t) \leqslant \dfrac{a(x)}{k} \, |t|\sp{p-2}t$}; \\
0,
& \text{if $|x| > R$ and $f(t) > \dfrac{a(x)}{k}\, 
|t|\sp{p-2}t$}.
\end{cases}
\end{equation*}
We also define the functions 
\( A, B \colon \mathbb{R}\sp{N} \to \mathbb{R} \) by
\begin{align*}
A(x) & = 
\begin{cases}
a(x), & \text{if $|x| \leqslant R$
or if $|x| > R$ and 
$f(u(x)) \leqslant \dfrac{a(x)}{k}\,u(x) $}; \\
\left( 1 - \dfrac{1}{k}\right)a(x),
& \text{if $|x| > R$ and $f(u(x)) 
> \dfrac{a(x)}{k}\,u(x)$},
\end{cases}
\end{align*}
and \( B(x) = b(x) \).

Considering these functions and
using \( v \in E \) as a test function, we have
\begin{align*}
0 & = \int\sb{\mathbb{R}\sp{N}} 
|\nabla u|\sp{p-2} \nabla u \cdot \nabla v \,\mathrm{d}x
+ \int\sb{\mathbb{R}\sp{N}} A(x) |u|\sp{p-2}u v 
\,\mathrm{d}x \\
& \qquad + \int\sb{\mathbb{R}\sp{N}} 
|\nabla u|\sp{q-2} \nabla u \cdot \nabla v \,\mathrm{d}x
+ \int\sb{\mathbb{R}\sp{N}} B(x) |u|\sp{q-2}u v 
\,\mathrm{d}x
- \int\sb{\mathbb{R}\sp{N}} H(x,u) v \,\mathrm{d}x \\
& = 
\int\sb{\mathbb{R}\sp{N}} 
|\nabla u|\sp{p-2} \nabla u \cdot \nabla v \,\mathrm{d}x
+ \int\sb{\mathbb{R}\sp{N}} a(x) |u|\sp{p-2}u v 
\,\mathrm{d}x \\
& \qquad + \int\sb{\mathbb{R}\sp{N}} 
|\nabla u|\sp{q-2} \nabla u \cdot \nabla v \,\mathrm{d}x
+ \int\sb{\mathbb{R}\sp{N}} b(x) |u|\sp{q-2}u v 
\,\mathrm{d}x
- \int\sb{\mathbb{R}\sp{N}} g(x,u) v \,\mathrm{d}x.
\end{align*}

From hypothesis~\eqref{hyp:f1}, 
for \( |t| \) small enough we have 
\( | H(x,t)| \leqslant |f(t)| 
\leqslant c\sb{1} |t|\sp{p\sp{*}-1} \); 
from hypothesis~\eqref{hyp:f2}, for \( |t| \) 
big enough we have
\( | H(x,t)| \leqslant |f(t)| 
\leqslant c\sb{2} |t|\sp{\tau - 1} \)  
with \( \tau \in (p, p\sp{*})\). 
Combining both cases we obtain 
\( | H(x,t)| \leqslant |f(t)| 
\leqslant c\sb{0} |t|\sp{p\sp{*}-1} \) 
for every 
\( t \in \mathbb{R}\sp{+} \) 
and for every \( \tau \in (p, p\sp{*}) \).
Then, it follows that 
\( |H(x,u)| \leqslant 
c\sb{0} |u(x)|\sp{\tau - p} |u(x)|\sp{p-1} 
= h(x) |u(x)|\sp{p-1} \), where we define
\( h(x) \equiv c\sb{0} |u(x)|\sp{\tau - p} \).

Direct computations show that 
\( h \in L\sp{r}(\mathbb{R}\sp{N}) \) for 
\( r = p\sp{*}/(\tau - p) \). Indeed,
\begin{align*}
\int\sb{\mathbb{R}\sp{N}} |h(x)|\sp{r}\,\mathrm{d}x 
& \leqslant 
c\sb{0}\sp{p\sp{*}/(\tau - p)}
\int\sb{\mathbb{R}\sp{N}} 
|u|\sp{p\sp{*}}\,\mathrm{d}x \\
& \leqslant
c\sb{0}\sp{p\sp{*}/(\tau - p)}
S\sp{p\sp{*}/p}
\left(
\int\sb{\mathbb{R}\sp{N}} 
|\nabla u|\sp{p}\,\mathrm{d}x
\right)\sp{p\sp{*}/p} \\
& \leqslant
c\sb{0}\sp{p\sp{*}/(\tau - p)}
S\sp{p\sp{*}/p}
\left\{
\splitfrac{
\displaystyle \int\sb{\mathbb{R}\sp{N}}
| \nabla u |\sp{p} \, \mathrm{d}x
+ \int\sb{\mathbb{R}\sp{N}}
a(x) | u |\sp{p} \, \mathrm{d}x }
{+\displaystyle \int\sb{\mathbb{R}\sp{N}}
| \nabla u |\sp{q} \, \mathrm{d}x
+ \int\sb{\mathbb{R}\sp{N}}
b(x) | u |\sp{q} \, \mathrm{d}x }
\right\}\sp{p\sp{*}/p} \\
& \leqslant 
c\sb{0}\sp{p\sp{*}/(\tau - p)}
S\sp{p\sp{*}/p}
\left\{ \left\| u \right\|\sb{1,p}\sp{p}
+ \left\| u \right\|\sb{1,q}\sp{q} 
\right\}\sp{p\sp{*}/p} 
< + \infty.
\end{align*}

In this way, any positive ground state solution 
\( u \in 
D\sp{1,p}(\mathbb{R}\sp{N}) \cap 
D\sp{1,q}(\mathbb{R}\sp{N}) \) 
to the auxiliary problem~\eqref{eq:auxprob} 
verifies the hypothesis of Lemma~\ref{lem:asbk}.
Concluding the argument, from inequality~\eqref{eq:Sob}
and from Lemma~\ref{lem:estimsol} we have 
\begin{align*} 
\left\| u \right\|\sb{L\sp{p\sp{*}}
(\mathbb{R}\sp{N})}
& \leqslant S\sp{1/p} \left\{
 \left\| u \right\|\sb{1,p}\sp{p} 
+ \left\| u \right\|\sb{1,q}\sp{q} 
\right\}\sp{1/p} 
\leqslant 
\left( \dfrac{Sdkp}{p-1} \right)\sp{1/p}.
\end{align*}
Finally, combining 
estimate~\eqref{eq:estimatelinf}
with the previous inequality we obtain
\begin{align*}
\| u \|\sb{L\sp{\infty}(\mathbb{R}\sp{N})} 
& \leqslant
M\sb{1} 
\max\left\{ 
\left\| u \right\|\sb{L\sp{p\sp{*}}(\mathbb{R}\sp{N})}, 
K, KL\sb{u}, 1 \right\} \\
& \leqslant
M\sb{1} 
\max\left\{
\left( \dfrac{Sdkp}{p-1} \right)\sp{1/p},
K, KL\sb{u}, 1 \right\} \\
& \equiv M\sb{2}, 
\end{align*}
where
\( M\sb{2} = M\sb{2}(N, p, q, r, 
a\sb{\infty}, b\sb{\infty}, \theta, c\sb{0}) \).
The lemma is proved.
\end{proof}

\begin{lem}
\label{lem:inequ}
Suppose that \( R\sb{0} \geqslant R > 1 \) and 
let 
\( u \in 
D\sp{1,p}(\mathbb{R}\sp{N}) \cap 
D\sp{1,q}(\mathbb{R}\sp{N}) \) 
be a positive ground state solution to the 
auxiliary problem~\eqref{eq:auxprob}. 
Then \( u \) verifies the inequality
\begin{equation*}
u(x) \leqslant
M\sb{2} \,
\dfrac{R\sp{(N-p)/(p-1)}}
    {|x|\sp{(N-p)/(p-1)}}
\end{equation*}
for every \( |x| \geqslant R > 1 \).
\end{lem}

\begin{proof}
Given \( R\sb{0} \geqslant R > 1 \), 
we define the function 
\( v \colon 
\mathbb{R}\sp{N}\backslash\{ 0 \} \to \mathbb{R} \)
by
\begin{align*} 
v(x) & \equiv M\sb{2} \, 
\dfrac{R\sb{0}\sp{(N-p)/(p-1)}}{|x|\sp{(N-p)/(p-1)}}.
\end{align*}
By hypothesis,
\( u \in 
D\sp{1,p}(\mathbb{R}\sp{N}) \cap 
D\sp{1,q}(\mathbb{R}\sp{N}) \) 
is a positive ground state solution to the 
auxiliary problem~\eqref{eq:auxprob}; therefore, 
we can apply Lemma~\ref{lem:estiminf} to deduce that
\( \| u \|\sb{L\sp{\infty}(\mathbb{R}\sp{N})}
\leqslant M\sb{2} \).
This implies that 
if \( |x| = R\sb{0} \), then 
\( \| u \|\sb{L\sp{\infty}(\mathbb{R}\sp{N})}
\leqslant v(x)\).
Now we define the function
\( w \colon \mathbb{R}\sp{N}\backslash\{ 0 \} \to 
\mathbb{R} \) by
\begin{align*}
w(x) = 
\begin{cases}
0,           & \text{if $|x| \leqslant R\sb{0}$}; \\
(u-v)\sp{+}, & \text{if $|x| \geqslant R\sb{0}$}. 
\end{cases}
\end{align*}
In this way, 
\( w \in D\sp{1,p}(\mathbb{R}\sp{N}) 
 \cap D\sp{1,q}(\mathbb{R}\sp{N})
\); moreover, \( w \in E \) because \( u, v \in E \).

To complete the proof of the lemma we will show that 
\( ( u-v )\sp{+} = 0 \) for \( |x| \geqslant R\sb{0} \).
To accomplish this goal we use the hypotheses on
the potential functions \( a \) and \( b \); 
we will also use the function \( w \in E \) 
as a test function to obtain
\begin{align}
\label{eq:graduw}
{} & 
\int\sb{\mathbb{R}\sp{N}} 
|\nabla u|\sp{p-2} \nabla u \cdot \nabla w \,\mathrm{d}x
+ \int\sb{\mathbb{R}\sp{N}} 
|\nabla u|\sp{q-2} \nabla u \cdot \nabla w 
\,\mathrm{d}x \nonumber  \\
& \qquad 
=  \int\sb{\mathbb{R}\sp{N}} g(x,u) w \,\mathrm{d}x
- \int\sb{\mathbb{R}\sp{N}} a(x) |u|\sp{p-2}u w 
\,\mathrm{d}x 
- \int\sb{\mathbb{R}\sp{N}} b(x) |u|\sp{q-2}u w 
\,\mathrm{d}x \nonumber \\
& \qquad 
=  \int\sb{\mathbb{R}\sp{N}\backslash B\sb{R\sb{0}}(0) 
\wedge f(t) \leqslant a(x)|t|\sp{p-2}t/k} 
f(u) w \,\mathrm{d}x
+ \int\sb{\mathbb{R}\sp{N}\backslash B\sb{R\sb{0}}(0)
\wedge f(t) > a(x)|t|\sp{p-2}t/k}
\dfrac{a(x)}{k}\, |u|\sp{p-2} u w 
\,\mathrm{d}x \nonumber \\
& \qquad \qquad 
- \int\sb{\mathbb{R}\sp{N}\backslash B\sb{R\sb{0}}(0)} 
a(x) |u|\sp{p-2}u w \,\mathrm{d}x 
- \int\sb{\mathbb{R}\sp{N}\backslash B\sb{R\sb{0}}(0)} 
b(x) |u|\sp{q-2}u w \,\mathrm{d}x \nonumber \\
& \qquad 
\leqslant 
\int\sb{\mathbb{R}\sp{N}\backslash B\sb{R\sb{0}}(0) 
\wedge f(t) \leqslant a(x)|t|\sp{p-2}t/k} 
\dfrac{a(x)}{k}\,|u|\sp{p-2} u w \,\mathrm{d}x 
\nonumber \\
& \qquad \qquad 
+ \int\sb{\mathbb{R}\sp{N}\backslash B\sb{R\sb{0}}(0) 
\wedge f(t) > a(x)|t|\sp{p-2}t/k} \dfrac{a(x)}{k}\, 
|u|\sp{p-2} u w \,\mathrm{d}x \nonumber \\
& \qquad \qquad 
- \int\sb{\mathbb{R}\sp{N}\backslash B\sb{R\sb{0}}(0)} 
a(x) |u|\sp{p-2}u w \,\mathrm{d}x 
- \int\sb{\mathbb{R}\sp{N}\backslash B\sb{R\sb{0}}(0)} 
b(x) |u|\sp{q-2}u w \,\mathrm{d}x \nonumber \\
& \qquad 
= \left( \dfrac{1}{k} -1 \right) 
\int\sb{\mathbb{R}\sp{N}\backslash B\sb{R\sb{0}}(0)} 
a(x)|u|\sp{p-2} u w \,\mathrm{d}x
- \int\sb{\mathbb{R}\sp{N}\backslash B\sb{R\sb{0}}(0)} 
b(x) |u|\sp{q-2}u w \,\mathrm{d}x \nonumber \\
& \qquad \leqslant 0
\end{align}
because \( u \) is a positive function and 
\( w \) is a nonnegative fuction, while 
\( k 
> 1 \).

Using the radially symmetric form of the operator 
\( \Delta\sb{m} u \), we have
\begin{align*}
\int\sb{\mathbb{R}\sp{N}\backslash B\sb{R\sb{0}}(0)} 
|\nabla v|\sp{m-2} \nabla v \cdot \nabla \phi 
\,\mathrm{d}x
& = 0 
\end{align*}
for \( m \in \{ p, q \} \) 
and for every function \( \phi \in E \). 
Therefore,
\begin{align}
\label{eq:gradvw}
{} & \int\sb{\mathbb{R}\sp{N}} 
|\nabla v|\sp{p-2} \nabla v \cdot \nabla w \,\mathrm{d}x
+ \int\sb{\mathbb{R}\sp{N}} 
|\nabla v|\sp{q-2} \nabla v \cdot \nabla w 
\,\mathrm{d}x \nonumber \\
& \qquad  = 
\int\sb{\mathbb{R}\sp{N}\backslash B\sb{R\sb{0}}(0)} 
|\nabla v|\sp{p-2} \nabla v \cdot \nabla w \,\mathrm{d}x
+ \int\sb{\mathbb{R}\sp{N}\backslash B\sb{R\sb{0}}(0)} 
|\nabla v|\sp{q-2} \nabla v \cdot \nabla w 
\,\mathrm{d}x \nonumber \\
& \qquad = 0.
\end{align}

Defining the subsets
\begin{align*}
\widetilde{A} & \equiv \{ x \in \mathbb{R}\sp{N} 
\colon |x| \geqslant R\sb{0} 
\text{ and } u(x) > v(x) \} \\
\intertext{and}
\widetilde{B} & \equiv \{ x \in \mathbb{R}\sp{N} 
\colon |x| < R\sb{0} 
\text{ or } u(x) \leqslant v(x) \},
\end{align*}
we have \( w(x) = u(x) - v(x) \) 
for \( x \in \tilde{A} \) 
and 
\( w(x) = 0 \) for \( x \in \tilde{B} \). 
Using inequality~\eqref{eq:graduw} and 
equation~\eqref{eq:gradvw} we get
\begin{align}
\label{eq:intgrada}
0 & \geqslant 
\int\sb{\mathbb{R}\sp{N}} 
|\nabla u|\sp{p-2} \nabla u \cdot \nabla w \,\mathrm{d}x
+ \int\sb{\mathbb{R}\sp{N}} 
|\nabla u|\sp{q-2} \nabla u \cdot \nabla w 
\,\mathrm{d}x \nonumber  \\
& \qquad 
- \int\sb{\mathbb{R}\sp{N}} 
|\nabla v|\sp{p-2} \nabla v \cdot \nabla w \,\mathrm{d}x
- \int\sb{\mathbb{R}\sp{N}} 
|\nabla v|\sp{q-2} \nabla v \cdot \nabla w 
\,\mathrm{d}x \nonumber \\
& = 
\int\sb{\widetilde{A}} 
\left[
|\nabla u|\sp{p-2} \nabla u 
- |\nabla v|\sp{p-2} \nabla v \right] 
\cdot (\nabla u - \nabla v )\, \mathrm{d}x \nonumber \\
& \qquad
+ \int\sb{\widetilde{A}} 
\left[
|\nabla u|\sp{q-2} \nabla u 
- |\nabla v|\sp{q-2} \nabla v \right] 
\cdot (\nabla u - \nabla v )\, \mathrm{d}x.
\end{align}

Denoting by 
\( \langle\cdot \, ,\cdot\rangle 
\colon \mathbb{R}\sp{N} \times \mathbb{R}\sp{N} 
\to \mathbb{R} \) 
the standard scalar product, given 
\( p \geqslant 2 \) there exists a positive constant 
\( c\sb{p} \in \mathbb{R}\sp{+} \) 
such that 
for every \( x,y \in \mathbb{R}\sp{N} \)
it is valid the inequality
\begin{align}
\label{eq:gte}
\langle |x|\sp{p-2}x - |y|\sp{p-2}y, x-y \rangle
\geqslant
c\sb{p} \, \| x-y\|\sp{p}
\end{align}
For the proof, we refer the reader to 
Simon~\cite{bib:simon}.
From inequalities~\eqref{eq:intgrada} 
and~\eqref{eq:gte} it follows that
\begin{align*}
\int\sb{\mathbb{R}\sp{N}} 
|\nabla w|\sp{p} \,\mathrm{d}x
+ \int\sb{\mathbb{R}\sp{N}} 
|\nabla w|\sp{q} \,\mathrm{d}x
& = 
\int\sb{\widetilde{A}} 
|\nabla u - \nabla v|\sp{p} \,\mathrm{d}x
\int\sb{\widetilde{A}} 
|\nabla u - \nabla v|\sp{q} \,\mathrm{d}x \\
& \leqslant
c\sb{p}\sp{-1} 
\int\sb{\widetilde{A}} 
\left[
|\nabla u|\sp{p-2}\nabla u - 
|\nabla v|\sp{p-2}\nabla v \right]
\cdot(\nabla u - \nabla v) \,\mathrm{d}x \\
& \qquad +
c\sb{q}\sp{-1} 
\int\sb{\widetilde{A}} 
\left[
|\nabla u|\sp{q-2}\nabla u 
- |\nabla v|\sp{q-2}\nabla v \right]
\cdot(\nabla u - \nabla v) \,\mathrm{d}x \\
& \leqslant 0.
\end{align*}

From this inequality we deduce
that each term on the left-hand side of the previous 
inequality must be zero, that is, 
\( w \) is constant in \( \mathbb{R}\sp{N} \). 
But we already know that \( w(x) = 0 \) in  the ball 
\( B\sb{R\sb{0}}(0) \); therefore, \( w(x) = 0 \) 
for every \( x \in \mathbb{R}\sp{N} \). 
This implies that 
\( (u-v)\sp{+} = 0 \) for \( |x| \geqslant R\sb{0} \) 
and \( u(x) \leqslant v(x) \) for every 
\( x \in \mathbb{R}\sp{N}\).
The proof of the lemma is complete.
\end{proof}

\section{Obtaining the solution of the original problem}
\label{sec:obtainsol}
In this section we finally show that the solution 
to the auxiliary problem~\eqref{eq:auxprob} obtained 
in section~\ref{sec:auxprob} is in fact a solution 
to problem~\eqref{eq:prob}. 

\begin{proof}[Proof of Theorem~\ref{thm:main}]
From Lemmas~\ref{lem:psbounded} and~\ref{lem:funcjps}, 
the auxiliary problem~\eqref{eq:auxprob} has a 
positive ground state solution 
\( u \in D\sp{1,p}(\mathbb{R}\sp{N}) 
\cap D\sp{1,q}(\mathbb{R}\sp{N}) \).
To accomplish our goal we need to show that for every 
\( x \in B\sb{R}\sp{c}(0) \) the function \( u \) 
verifies the inequality
\begin{equation*}
f(u) \leqslant
\dfrac{a(x)}{k}\, |u|\sp{p-2}u.
\end{equation*}

From Lemma~\ref{lem:inequ} and by the first inequality 
in~\eqref{eq:conshypf1f2}, 
if \( |x|\geqslant R \), then
\begin{align*}
\dfrac{f(u)}{|u|\sp{p-2}u}
& \leqslant c\sb{0} \, 
\dfrac{|u|\sp{p\sp{*}-2}}{|u|\sp{p-2}} 
\leqslant c\sb{0}
\Bigg\{
M\sb{2}\, 
\dfrac{\big( R\sp{p/(p-1)}\big)\sp{(N-p)/p} 
}{\big( |x|\sp{p/(p-1)} \big)\sp{(N-p)/p}}
\Bigg\}\sp{p\sp{*}-p} 
= c\sb{0}  \,
M\sb{2}\sp{p\sp{*}-p} \,
\dfrac{R\sp{p\sp{2}/(p-1)}}{|x|\sp{p\sp{2}/(p-1)}}.
\end{align*}

Now we define the constant
\begin{equation*}
\Lambda\sp{*} \equiv 
c\sb{0} k
M\sb{2}\sp{p\sp{*}-p}.
\end{equation*}
Considering
\( \Lambda \geqslant \Lambda\sp{*} \),
it follows from the hypothesis~\eqref{hyp:v3} that
\begin{align*}
\dfrac{f(u)}{|u|\sp{p-2}u} 
& \leqslant
\dfrac{\Lambda\sp{*}}{k}\,
\dfrac{R\sp{p\sp{2}/(p-1)}}{|x|\sp{p\sp{2}/(p-1)}} 
\leqslant
\dfrac{\Lambda}{k}\,
\dfrac{R\sp{p\sp{2}/(p-1)}}{|x|\sp{p\sp{2}/(p-1)}} 
\leqslant
\dfrac{a(x)}{k}.
\end{align*}
The proof of the theorem is complete.
\end{proof}

\end{document}